\documentclass[11pt]{amsart}
\usepackage{amsfonts, amsmath, amssymb, amsthm}
\usepackage{epic}
\usepackage{xypic}
\usepackage{graphicx,color}
\usepackage{changebar}

\newcommand{\sexp}{{\rm sexp} }

\newtheorem{lemma}{Lemma}
\newtheorem{theorem}{Theorem}

\newtheorem{proposition}{Proposition}
\newtheorem{corollary}{Corollary}
\newtheorem{definition}{Definition}

\newtheorem{remark}{Remark}
\newtheorem{fact}{Fact}  

\newtheorem{example}{Example}

\newtheorem{Acknow}{Acknowledgement}

\newcommand{\nth}[1]{$#1 \hbox{-th }$}

\newcommand{\ord}{{\rm ord} }
\newcommand{\eord}{{\rm ess. ord} }

\newcommand{\Z}{\mathbb{Z}}

\newcommand{\rg}[1]{\mathbf{#1}}
\newcommand{\eu}[1]{\mathfrak{#1}}
\newcommand{\re}[1]{\mathbf{\mathfrak{#1}}}
\newcommand{\id}[1]{\mathcal{#1}}
\newcommand{\Gal}{\hbox{Gal}}

\newcommand{\Ker}{\hbox{ Ker }}
\newcommand{\Coker}{\hbox{ Coker }}

\newcommand{\disc}{\mbox{disc}}
\newcommand{\prk}{p\hbox{-rk} }
\newcommand{\eprk}{ess. \ p\hbox{-rk}}
\newcommand{\zprk}{\Z_p\hbox{-rk}}
\newcommand{\zrk}{\Z\hbox{-rk} }

\newcommand{\rf}[1]{(\ref{#1})}
\newcommand{\ran}{\rangle}
\newcommand{\lan}{\langle}
\newcommand{\rad}{\hbox{rad}}
\newcommand{\Rad}{\hbox{Rad}}

\newcommand{\pinf}{^{1/p^{\infty}}}

\newcommand{\Norm}{\mbox{\bf N}}
 
\newcommand{\lchooses}[2]{\left( \frac{#1}{#2 } \right)}

\newcommand{\F}{\mathbb{F}}
\newcommand{\KP}{\mathbf{\Phi}}

\newcommand{\A}{\mathbb{A}}

\newcommand{\B}{\mathbb{B}}
\newcommand{\K}{\mathbb{K}}

\newcommand{\LK}{\mathbb{L}}
\newcommand{\KH}{\mathbb{H}}

\newcommand{\KL}{\mathbb{L}}

\newcommand{\M}{\mathbb{M}}
\newcommand{\Q}{\mathbb{Q}}

\newcommand{\C}{\mathbb{C}}
\newcommand{\N}{\mathbb{N}}

\def\ra{\rightarrow}

\begin{document}
{\obeylines \small
\vspace*{-1.0cm}
\hspace*{3.0cm}Par les gosses battus, par l'ivrogne qui rentre
\hspace*{3.0cm}Par l'\^{a}ne qui re\c{c}oit des coups de pied au ventre
\hspace*{3.0cm}Et par l'humiliation de l'innocent ch\^{a}ti\'{e}
\hspace*{3.0cm}Par la vierge vendue qu'on a d\'{e}shabill\'{e}e
\hspace*{3.0cm}Par le fils dont la m\`{e}re a \'{e}t\'{e} insult\'{e}e
\vspace*{0.2cm} 
\hspace*{4.0cm} Je vous salue, Marie\footnote{Francis Jammes: {\em Pri\`ere}. Music by Georges Brassens}
\vspace*{0.4cm}
\hspace*{4.5cm} {\it To Theres and Seraina, to the memory of Marta }
\vspace*{0.5cm}
\smallskip}
\title[Leopoldt's Conjecture] {On CM $\Z_p$-extensions and the Leopoldt conjecture for CM fields} \author{Preda
  Mih\u{a}ilescu} \address[P. Mih\u{a}ilescu]{Mathematisches Institut
  der Universit\"at G\"ottingen}
\email[P. Mih\u{a}ilescu]{preda@uni-math.gwdg.de}
\date{Version 1.0 \today}
\vspace{0.5cm}
\begin{abstract}
  We show that Leopoldt's conjecture holds in CM fields. For the proof
  we construct a $CM \Z_p$-extension of some CM field in which the
  Leopoldt conjecture is supposed to fail, and using the classes of
  primes which are completely split in this extension, we derive an
  contradiction. The method of proof can be described as a {\em
    stability check of $\Lambda$-modules under deformations in Thaine
    shifts}.
 \end{abstract}
\maketitle
\tableofcontents
\section{Introduction}
Let $p$ be an odd rational prime and $\K/\Q$ be a finite galois CM
extension with group $\Delta$, of which we shall assume that it
contains the \nth{p} roots of unity. We denote by $\K_{\infty}$ the
cyclotomic $\Z_p$-extension of $\K$ and $\KL$ any other $\Z_p$
extensions. The intermediate fields will be $\K_n$, resp. $\KL_n$.

If $\KL/\K$ is an arbitrary $\Z_p$-extension, with group $\Gamma \cong
\Z_p$, generated by $\tau \in \Gamma$ as a topological generator, the
Iwasawa algebra is $\Lambda = \Z_p[[ T ]], T = \tau - 1$. The
intermediate fields are $\KL_n \subset \KL$ and, if $\KL =
\K_{\infty}$ is the cyclotomic $\Z_p$-extension, then we always assume
that $\K_n = \KL_n$ contains exactly the \nth{p^n} roots of unity, but
not the \nth{p^{n+1}} ones; this can be achieved by an adequate
numeration, at least for some sufficiently large $n$.
 
We write $\tau_n = (T+1)^{p^{n-k}}$ for the power of $\tau$ that
generates the fixing group of $\K_n$, and $\omega_n = \tau_n - 1;
\nu_{n+j,n} = \omega_{n+j}/\omega_n, j > 0$. The $p$-parts of the
ideal class groups of $\K_n, \KL_n$ are $A_n = A(\K_n), A(\KL_n)$ and
$A = A(\K) = \varprojlim_n A(\K_n), A(\KL) = \varprojlim_n A(\LK_n)$;
the groups $A'_n, A'$ are defined like $A_n, A$, with respect to the
ideal classes of the $p$-integers of $\K_n, \KL_n$. The groups $B_n
\subset A_n$ are the maximal subgroups generated by classes containing
ramified primes above $p$ and $B = \varprojlim_n B_n$, while $A' =
A/B$. We note that the base field $\K$ can be modified within the same
$\Z_p$-extension, by replacing $\K$ with $\K_n$, say. As a
consequence, the Iwasawa algebra may become $\Lambda' = \Z_p[[
\omega_n ]]$.
\subsection{Notation and questions}
For arbitrary number fields $\K$, the cyclotomic $\Z_p$-extension
is denoted by $\K_{\infty}$.
In some parts of this paper we shall consider also a further $\Z_p$-extension
 $\KL/\K$, setting some important additional conditions on the
intersection $\KL \cap \K_{\infty}$. In such context we encounter at
least two different Iwasawa algebras; additional variation can result
from changing the base field as mentioned above. We restrict our
introductory notation and remarks to the above; the precise choices
and adequate notations for combined extensions will be introduced in
the given context. Here we still restrict to the context of one single,
not necessarily cyclotomic, $\Z_p$-extension $\KL/\K$ and introduce
some additional concepts and notations.

\begin{definition}
\label{initdef}
For $f \in \Z_p[ T ]$ a distinguished polynomial and $M$ an additively
written $\Lambda$-torsion module, we write
\begin{eqnarray}
  \label{torsion}
  M(f) & = & \{ a \in A \ : \ \exists m: f^m(T) a = 0 \} \quad \hbox{ and } \nonumber \\
  M[ f ] & = & \{ a \in A \ : \ f(T) a = 0 \}.
\end{eqnarray}
If $X$ is a finite abelian $p$-group, its exponent is $\exp(X) =
\min\{ p^m \ : \ p^m X = 0 \}$.  The {\em subexponent} is the smallest
size of a cyclic direct summand in $X$, thus $\sexp(X) = \min \{
\ord(x) \ : \ x \in X \setminus p X \}$.

We let $F(A) \subset A$ be the maximal finite $\Lambda$-submodule of
$A$, while $A^{\circ}$ denotes the $\Z_p$-torsion submodule, so $F(A)
\subset A^{\circ}$.
\end{definition}
If $\K$ is a number field, we denote its units by $E(\K) =
\id{O}^{\times}(\K)$. Dirichlet's unit theorem states that, up to
torsion made up by the roots of unity $W(\K) \subset \K^{\times}$, the
units $E = \id{O}(\K)^{\times}$ are a free $\Z$ - module of $\Z$ -
rank $r_1 + r_2 - 1$. As usual, $r_1$ and $r_2$ are the numbers of
real, resp. pairs of complex conjugate embeddings $\K \hookrightarrow
\C$. We consider the set $P = \{ \wp \subset \id{O}(\K) : (p) \subset
\wp \}$ of distinct prime ideals above $p$ and let
\[ \re{K}_p = \re{K}_p(\K) = \prod_{\wp \in P } \K_{\wp} = \K \otimes_{\Q} \Q_p
\] 
be the product of all completions of $\K$ at primes above $p$.  Let
$\iota : \K \hookrightarrow \re{K}_p$ be the diagonal embedding. We
write $\iota_{\wp}(x)$ for the projection of $\iota(x)$ in the
completion at $\wp \in P$. If $y \in \re{K}_p$, then $\iota_{\wp}(y)$
is simply the component of $y$ in $\K_{\wp}$. If $U \subset
\re{K}_p^{\times}$ is the group of units, thus the product of local
units at the same completions, then $E$ embeds diagonally via $\iota :
E \hookrightarrow U$.

Let $\overline{E} = \overline{\iota(E)} = \bigcap_{ n > 0} \iota(E)
\cdot U^{p^n} \subset U$ be the $p$-adic closure of $\iota(E)$; this
is a $\Z_p$ - module with $\zprk(\overline{E}) \leq \zrk(E) = r_1 +
r_2 - 1$. The difference
\[ \id{D}(\K) = \zrk(E) - \zprk(\overline{E}) \] is called the
\textit{Leopoldt defect}. The defect is positive if, in the idelic
embedding, units which are $\Z$-independent, are related $p$-adically.
Equivalently, if the $p$ - adic regulator of $\K$ vanishes.

Leopoldt suggested in \cite{Le} that $\id{D}(\K) = 0$ for all number
fields $\K$. This conjecture of Leopoldt was proved for abelian
extensions by Brumer \cite{Br} in 1967, using a result of Ax \cite{Ax}
and a local version of Baker's linear forms in logarithms
\cite{Ba}. It is still open for arbitrary non abelian
extensions. Since 1967 various attempts have been undertaken for
extending the results of \cite{Br} to non abelian extensions, using
class field theory, Diophantine approximation or both. The following
very succinct list is intended to give an overview of various
approaches, rather than being an extensive list of results on
Leopoldt's conjecture. In \cite{Gr}, Greenberg notes a relation
between the Leopoldt Conjecture and a special case of the Greenberg
Conjecture: he shows that Leopoldt's Conjecture implies that $A(T)$
(see \S 1.1. for the definitions) is finite for totally real fields,
i.e. the Greenberg Conjecture holds for the $T$ - part.

Emsalem, Kisilevsky and Wales \cite{EKW} use group representations and
Baker theory for proving the Conjecture for some small non abelian
groups; this direction of research has been continued in some further
papers by Emsalem or Emsalem and coauthors. Jaulent proves in
\cite{Ja2} the Conjecture for some fields of small discriminants,
using the \textit{phantom} field $\KP$ which we define in the
Appendix. Currently the strongest result based on Diophantine
approximation was achieved by Waldschmidt \cite{Wal}, who proved that
if $r$ is the $\Z$ - rank of the units in the field $\K$, then the
Leopoldt defect satisfies $\id{D}(\K) \leq r/2$.

The connection of Leopoldt's conjecture to class field theory was
already noted by Iwasawa in \cite{Iw}. He shows that if $\Omega(\K)
\supset \K_{\infty}$ is the maximal $p$-abelian $p$-ramified extension
of $\K$, then $\Gal(\Omega(\K)/\K) \sim \Z_p^n$, where $n =
r_2+1+\id{D}(\K)$; the proof of this fact is in any text book on
cyclotomy and Iwasawa theory. For CM extensions $\K$, the
contraposition of the conjecture herewith reduces to the statement
that $\K^+$ has more $\Z_p$-extensions than just the cyclotomic
one. It is this assumption which we shall use and lead to a
contradiction. If Leopoldt fails thus for $\K^+$ and if $\KL^+/\K^+$
is a further $\Z_p$-extension, it will be totally real and $\KL =
\KL^+ \cdot \K$ will be a CM extension. Starting from this fact, we
prove in this paper:
\begin{theorem}
\label{main}
For odd primes $p$, the Leopoldt defect vanishes in arbitrary CM
extensions $\K$.
\end{theorem}

We shall use the following notations, for arbitrary fields $\rg{K}$:
the maximal unramified $p$-abelian extension is $\KH(\rg{K})$. If
$\rg{K}$ is CM, then complex conjugation acts on galois groups, and
maximal extensions split naturally in plus and minus parts: for
instance, $\KH^+(\rg{K})$ is the subfield fixed by
$\Gal(\KH/\rg{K})^{1-\jmath}$, with $\jmath \in \Gal(\rg{K}/\Q)$ the
restriction of complex conjugation to this field. The maximal
$p$-abelian, $p$-ramified extension is denoted by $\Omega(\rg{K})$ and
$\M(\rg{K}) \subset \Omega(\rg{K})$ is the product of all
$\Z_p$-extensions of $\rg{K}$, while $\rg{K}_{\infty}/\rg{K}$ is the
cyclotomic $\Z_p$-extension.

For CM extensions we define $\M^+(\rg{K})
= \M(\rg{K}^+) \cdot \rg{K}$. The Leopoldt conjecture holds for
$\rg{K}$ iff $\M^+(\rg{K}) = \rg{K}_{\infty}$, and in general we have
the rank equality
\[ \zprk(\Gal(\M^+(\rg{K})/\rg{K})) = \id{D}(\rg{K}) + 1 . \] The term
$1$ on the right hand side stands for the cyclotomic $\Z_p$-extension,
which is the only one which is expected to exist, if Leopoldt's
conjecture is true. We note that $\M^+(\rg{K})$ is also the product of
all CM $\Z_p$-extensions of $\rg{K}$ and Leopoldt's conjecture thus
claims that $\rg{K}$ has only one CM $\Z_p$-extension, namely the
cyclotomic one. The CM property of $\Z_p$-extensions plays a crucial
role in our proof.

\begin{remark}
\label{shift}
\begin{itemize}
\item[ A. ] If $\K_{-1}$ is a field for which $\id{D}(\K_{-1}) > 0$,
  then it is known that the same holds for arbitrary finite algebraic
  extensions $\K/\K_{-1}$; this is noted, for instance, by Laurent in
  the introduction to \cite{Lau}. We shall use negative indices for
  designing number fields which will first be enlarged for certain
  purposes, before considering actual $\Z_p$-extensions. We thus keep
  the notation $\K$ for base fields of the $\Z_p$-extensions of
  interest.

  It follows from the fact that the linear relations between
  $\Z$-generators of the units of $E(\K_1)$, which arise upon $p$-adic
  completion, will be preserved under the embedding into the units
  $E(\K)$.

\item[ B. ] If $\K$ is a CM extension containing the \nth{p^m} roots
  of unity and $\KL^+/\K^+$ is a $p$-ramified cyclic extension of
  degree $p^m$, then there is a class $a \in A^-(\K)$ such that $\KL^+
  \cdot \K = \K[ a^{1/p^m} ]$ in the sense that for each $\eu{A} \in
  a$ there is an $\alpha \in \eu{A}^{(1-\jmath) p^m}$ with $\KL = \K[
  \alpha^{1/p^m} ]$. We shall use this notation for explicit Kummer
  extensions throughout the paper.

\end{itemize}
\end{remark}
The following result has been proved independently by Babaicev
\cite{Bab} and Monsky \cite{Mo}:
\begin{theorem}
\label{mubound}
Let $\K$ be a number field and let $\M$ be the product of all of its
$\Z_p$-extensions.  Then there is an absolute bound $B = B(\K) > 0$
such that for all $\Z_p$-subextensions $\K \subset \KL \subset \M$, we
have $\mu(\KL) \leq B(\K)$, where $\mu(\KL)$ is Iwasawa's $\mu$
constant for $\KL$.
\end{theorem}
It will be used in our proof for bounding the $\mu$-constant for CM
$\Z_p$-extensions.

\subsection{Plan of the proof}
The proof is inspired by a construction of Iwasawa for showing that
there exist $\Lambda$-extensions with $\mu > 0$. We assume that $K$ is
some CM extension for which the Leopoldt defect does not vanish, and
create first an auxiliary extension $\K \supset K$ which also has
positive Leopoldt defect, together with some additional properties. We
then construct a CM $\Z_p$-extension $\KL/\K$ with $\KL \cap
\K_{\infty} = \K_N$ for some large $N$, and such that there is a prime
$\eu{q} \in a_N \in A^-(\K_N)$ which is totally split in $\KL/\Q$. The
class $a_N$ is induced by ideal lifts from an abelian extension
$\rg{k}$ that we can control. We then define $\F \subset \Q[ \zeta_q
]$ the subfield of degree $p$, with $q$ the rational prime above
$\eu{q}$ and let $\KL' = \KL \cdot \F$, an extension in which the
split primes above $\eu{q}$ will ramify. The explicit construction is
described in the first section of Chapter 3. If $x \in A(\KL)$ for
some $\Z_p$-extension $\KL/\K$, we distinguish the case when $\Lambda
x$ is infinite of finite $p$-rank (the $\lambda$-type) or infinite of
bounded order (the $\mu$-type). Modules which can be split into cyclic
submodules of the two kinds are called decomposed\footnote{See below
  for the formal complete definition.}.  We let $\eu{Q}_n \subset
\KL'_n$ be norm coherent ramified primes above the split primes
$\eu{q}_n \subset \KL_n$, above $\eu{q}$. Letting $b_n = [
\eu{Q}_n^{1-\jmath} ]$, we show that the sequence of classes $b =
(b_n)_{n \in \N} \in A^-(\KL')$ must necessarily be indecomposed --
otherwise a contradiction is easily obtained.

Growth and decomposition of $\Lambda$-modules play an important role
in our proof, and they are investigated at length in the second
chapter. We show in particular that we may choose $\K$ such that $T
A^-(\KL)$ is a decomposed module, for all CM - $\Z_p$-extensions
$\KL/\K$. Thus, the sequence is {\em not far} from being decomposed,
and in fact we have $T b = b_{\lambda} + b_{\mu}$. The ramification
conditions will in fact imply that $b_{\lambda} \in
\iota(A^-(\KL))$. This particular condition will lead to a
contradiction with the choice of $\rg{k}$, a contradiction which shows
that the extension $\KL$ cannot exist.

The CM property of our extensions has an important contribution to the
simplicity of the proof, in that capitulation is reduced to a
$\Z_p$-cyclic, well understood submodule. This simplifications are
treated in \S 2.4. The structure of $b$ and its decomposition is
governed by consequences of the construction and the Hasse Norm
Principle. These consequences are derived in \S 3.2.

\section{Growth, stability and decomposition of
 $\Lambda$-modules}
In this section $\KL/\K$ is an arbitrary $\Z_p$-extension of the
base-field $\K$, which is assumed to be galois over $\Q$ and contain
the \nth{p} roots of unity, for simplicity. In addition, the primes
above $p$ are assumed to be completely ramified in $\KL/\K$. In this
section we develop several properties concerning the growth and
stabilization of $\Lambda$-modules along intermediate extensions of a
$\Z_p$-extension, as well as some properties of the module of
decomposed elements. These properties will be used in the next chapter
in order to define a particular tower of extensions with respect to
which we shall perform the proof of the main Theorem. In both
chapters, base fields will be denoted by $\K$ or some related
notation, and we consider one or more $\Z_p$-extensions thereof.  In
particular, the precise properties of the base field $\K$ for which we
perform the final proof, will be only given in the third chapter.

The Iwasawa algebra is defined like in the introduction and we recall
that $A$ is a finitely generated $\Lambda$-torsion module. We
associate elementary modules to $A$ as follows:
\begin{definition} 
\label{elmod}
Let $A$ be a finitely generated $\Lambda$-torsion module and
$\id{E}(A) = \id{E}(A)_{\lambda} \bigoplus \id{E}(A)_{\mu}$ be an {\em
  elementary $\Lambda$-module}, with $\id{E}(A)_{\mu} = \bigoplus_{i =
  1}^m \Lambda/(p^{e_i})$ and $\id{E}(A)_{\lambda} = \bigoplus_{j=1}^n
\Lambda/(f_j^{e_j})$. If $\id{E}(A) \sim A$ are pseudoisomorphic, we
say that the {\em elementary module} $\id{E}(A)$ is {\em associated to
  $A$}. The $\mu$-part $\id{E}(A)_{\mu}$ is uniquely determined by
$A$, while the distinguished polynomials $f_j \in \Z_p[ T ]$ and their
exponents $e'_j$ occurring in $\id{E}(A)_{\lambda}$ can
vary.\footnote{They can be fixed by assuming they are
  irreducible. However, such a choice can result in an increase of
  either kernel or cokernel, which might be undesirable in certain
  cases.}
\end{definition}

The following notions are connected to $\Lambda$-modules:
\begin{definition}
\label{type}
Let $\KL/\K$ be a $\Z_p$-extension and $\Lambda$ the associated
Iwasawa algebra. Let $N$ be some finitely generated $\Lambda$-torsion
module. We say that $a \in N$ is of {\em $\lambda$-type}, if $\Lambda
a$ is infinite of finite $p$-rank and of {\em $\mu$-type} if $\Lambda
a$ is infinite of finite order. Finally $a$ is of finite type, if
$\Lambda a$ is finite. Accordingly, $N$ is of one of the three types,
if it is generated by elements of one of the three types. Note that
modules of $\lambda$ or $\mu$-type can contain finite submodules.

The maximal finite $\Lambda$-module is $F(N)$ while $M(N) :=
N^{\circ}$ is its $\Z_p$-torsion submodule.  Note that $M(N)$ is at
the same time the module of all elements which are either of $\mu$- or
of finite type.  We let $L(N) = \{ x \in N \ : \ \prk(\Lambda x ) <
\infty \}$, the module of all elements that are either of $\lambda$-
or of finite type.

An element $x \in A(\KL)$ is {\em decomposable}, if there are
$x_{\lambda} \in L(A), x_{\mu} \in M(A)$, such that $x = x_{\lambda} +
x_{\mu}$. It is indecomposable otherwise. If $x$ is decomposable and
$x = x_{\lambda}+x_{\mu} = x'_{\lambda} + x'_{\mu}$ are two
decompositions, then $x'_{\lambda} - x_{\lambda} = -(x'_{\mu} -
x_{\mu}) \in F(A)$. The submodule
\[ D := \{ x = y + z \ : \ y \in L, z \in M \} \subset A \] is the
module of decomposable elements.
\end{definition}

If $\KH/\KL$ is the maximal abelian unramified extension of $\KL$, we
denote the Artin map by $\varphi : A(\KL) \ra \Gal(\KH/\KL)$.  It is
not difficult to see that $[ A : D ] < \infty$. Indeed, if $\psi :
\id{E}(A) \ra A$ is a pseudoisomorphism, then $\Ker(\psi) = 0$ since
the kernel is finite and $\id{E}(A)$ has no finite submodule. Thus $D'
:= \psi(\id{E}(A)) \subset A$ is a decomposed submodule, thus $D'
\subset D$. Since $\psi$ is a pseudoisomorphism, the cokernel is
finite so $\Coker(\psi) = A/D'$ is finite. A fortiori, $[ A : D ] \leq
[ A : D' ] < \infty$, so $A/D$ is finite.

If $x \in A \setminus D$, the $L$- and the $D$-orders of $x$ are,
respectively
\begin{eqnarray}
\label{dords}
\ell(x) & = & \min\{  j > 0 \ : \ p^{j} x \in L \}, \quad \hbox{and} \\
\delta(x) & = & \min\{  k > 0 \ : \ p^{k} x \in D \} \leq \ell(x). \nonumber
\end{eqnarray}

We may associate the modules to the extension, writing $L(\KL),
M(\KL), F(\KL)$ and note that the canonical submodules $L(\KL), M(\KL)
\subset A(\KL)$ verify
\begin{eqnarray}
\label{decolm}
D(\KL) = L(\KL)+M(\KL), \quad L(\KL) \cap M(\KL) = F(\KL)
\end{eqnarray}
Throughout the paper, unless otherwise specified, the distinguished
polynomial $F(T) \in \Z_p[ T ]$ will denote the minimal annihilator
polynomial of $L(\KL)$, i.e. the least common multiple of the minimal
annihilators $f_a(T) \in \Z_p[ T ]$ of all elements $a \in L(A)$.  For
$x \in M$ the order is naturally defined by $\ord(x) = \min\{ p^k \ :
\ p^k x = 0 \}$ and the {\em essential order} is $\eord(x) = \ord(T^j
x)$ for all but possibly finitely many $j \geq 0$. If $\psi : M \ra
\id{E}(M)$ is a pseudoisomorphism, then $\eord(x) = \ord(\psi(x))$,
since $\id{E}(M)$ has no finite submodules.
\begin{lemma}
\label{mpart}
Let $x \in M$ and $p^k = \ord(x), \ p \leq p^l = \eord(x) \leq
\ord(x)$.  Then for any $g \in \Lambda$ we have $g x = 0 \Rightarrow g
\equiv 0 \bmod p^l$ and there is a distinguished polynomial $G(T) \in
\Z_p[ T ]$ such that $G(T) A \subset M$, while $\ord(y) = \eord(y)$
for all $y \in G(T) A$.

If $x \in A \setminus D$ and $p^{\delta (x)} x = c + z, c \in L, z \in
M$, where $\delta(x)$ is the order defined in \rf{dords}, then $c \not
\in p L$.
\end{lemma}
\begin{proof}
  We fix a pseudoisomorphism $\psi : M \ra \id{E}(M)$.  By definition
  of the essential order, $p^l = \eord(\psi(x))$ and if $g(T) x = 0$
  then $\psi(g(T) x) = g(T) \psi(x) = 0$ thus $g(T) \equiv 0 \bmod
  \eord(x)$, as claimed.

  We show the existence of $G(T)$ in two steps. First, if $F(T)$ is a
  distinguished polynomial that annihilates $L$, then $F(T) A \subset
  M$.  If $g(T)$ is a distinguished polynomial that annihilates
  $\Ker(\psi) = F(A)$, then the restriction $\psi : g(T) M
  \hookrightarrow \id{E}(M)$ is injective, so $\ord(y) = \eord(y)$ for
  all $y \in g M$.  We may thus choose $G(T) = F(T) g(T)$ as a
  polynomial satisfying the second claim of the lemma.

  Let now $x \in A \setminus D$ and assume the final claim is false,
  so $p^{\delta(x)} x = c + u = p \gamma + u, \gamma \in L$. Let $y :=
  p^{\delta(x)-1} x$. Then $p (y-\gamma) = u \in M$ and it follows
  that $u' := y - \gamma \in M$ too. But then $p^{\delta(x)-1} x = y =
  \gamma + u' \in D$, in contradiction with the minimality of
  $\delta$. The hypothesis $c \in p L$ was thus false, which completes
  the proof of the lemma.
\end{proof}

We introduce the distances $d_n : A \times A \ra \N$ as follows: let
$x, z \in A$; then
\[ d_n(x, z) := \prk(\Lambda (x_n - z_n)); \quad d_n(x) = \prk(\Lambda
x_n) . \] We obviously have $d_n(x,z) \leq d_n(x,y)+d_n(x,z)$ and
$d_n(x) \geq 0$ with $d_n(x) = 0, \forall n > 0$, for the trivial
module. Also, if $f \in \Z_p[ T ]$ is some distinguished polynomial of
degree $\phi = \deg(f)$, then $d_n(x)-\phi \leq d_n(f x) \leq d_n(x)$
for all $x \in A$. We shall write $d(x, y) = \lim_n d_n(x,y)$. We may
also write $d(u,v) = \prk(\Lambda (u-v))$ if $u, v \in A_k$, but we
know no lifts of $u$ and/or $v$ to $A$: the difference consists here
in the fact, that $u, v$ appear as individual elements of $A_k$,
rather then elements of a given norm coherent sequence.  The simplest
fact about the distance is
\begin{fact}
\label{dlam}
Let $x, z \in A$ be such that $d_n(x, z) \leq N$ for some fixed bound
$N$ and all $n > 0$. Then $x-z \in L$ and $N \leq \ell := \prk(L)$.
For every fixed $d \geq 2 \prk(L)$ there is an integer $n_0(d)$ such
that for any $x \in A \setminus L$ and $n > n_0$, if $d_n(x) \leq d$
then $x \in \nu_{n,n_0} A + F$.
\end{fact}
\begin{proof}
  The element $y = x-z$ generates modules of bounded rank, so it is
  neither of $\mu$-type nor indecomposed. Thus $y \in L$ and
  consequently $d_n(y) \leq \prk(L_n) \leq \ell$ for all $n$, which
  confirms the claim.
  
  For the second claim, note that if $x \not \in L$, then $d_n(x) \ra
  \infty$, so the boundedness of $d_n(x)$ becomes a strong constraint
  for large $n$. Next we recall that $F(T) x \in M$ and since
  $d_n(F(T) x) \leq d_n(x)$, we may assume that $x \in M$. Now $d_n(x)
  \leq d$ implies the existence of some distinguished polynomial $h
  \in \Z_p[ T ]$ with $\deg(h) = d$ and such that $h(T) x_n = 0$. The
  exponent of $M$ is uniformly bounded by $p^{B}$, as a consequence of
  the Theorem of Babaicev -- Monsky, so there is a finite set $\id{H}
  \subset \Z_p[ T ]$ from which $h$ can take its values. Let now $n_0$
  be chosen such that $\nu_{n_0,1} \in (h(T), p^{B})$ for all $h \in
  \id{H}$. Since $m$ depends on $\K$, it follows that $n_0$ only
  depends on $\K$ too. Then $h(T) x_n = 0$ implies $\nu_{n_0,1} (x_n)
  = 0$ and thus, by Iwasawa's Theorem 6 (see also \rf{an} and Lemma
  \ref{ydef} below) there is a $z \in A$ such that $\nu_{n_0,1} (x) =
  \nu_{n,1} (z)$.  Consequently $\nu_{n_0,1}( x - \nu_{n,n_0} z) = 0$
  and thus $x = \nu_{n,n_0}( z ) + y$ for some $y$ and the claim
  follows if we show that $y := x - \nu_{n,n_0}(z) \in F$. Since
  $\nu_{n_0,1} y = 0$, it follows that $y \in L$ and the result of
  Sands in Lemma \ref{sa} in the Appendix implies that $\Lambda y$
  must be finite, so $y \in F$, as claimed.

  We now show that for fixed $h \in \id{H}$ there is some even $m$
  such that $\nu_{m,1} \in (h, p^{ \mu})$. Let the Euclidean division
  yield
  \[ \nu_{m,1} = q_m(T) h(T) + r_m(T), \] and let $\xi_i \in
  \overline{\Q}_p$ be the zeroes of $h$. Then
  \[ r_m(\xi_i) = \nu_{m,1}(\xi_i) = \frac{(1+\xi_i)^{p^m}-1}{\xi_i} =
  O(p^{m/2}) + O(\xi_i^{p^{m/2}}). \] Since $\id{H}$ is finite, there
  is some lower bound $\delta$ with $v_p(\xi_i) \geq \delta$ and the
  above identity shows that $v_p(r_m) \ra \infty$ with diverging $m$,
  so one can choose $m$ sufficiently large, such that $v_p(r_m) > \mu$
  and thus $r_m \equiv 0 \bmod p^{\mu}$, so $\nu_{m,1} \in (h(T),
  p^{\mu})$. This can be achieved for all $h \in \id{H}(T)$ and the
  claim is satisfied by choosing $n_0 = \max_{h \in \id{H}} (m(h))$ .
  \end{proof}

  The arguments of this chapter will take repeatedly advantage of the
  following elementary Lemma\footnote{I owe the proof of the Lemma to
    Cornelius Greither, who provided an elegant simplification of my
    original proof.}:
\begin{lemma}
\label{ab}
Let $A$ and $B$ be finitely generated abelian $p-$groups denoted
additively, and let $N: B\ra A$, $\iota:A\ra B$ be two $\Z_p$ - linear
maps such that:
\begin{itemize}
\item[1.] $N$ is surjective.
\item[2.] The $p-$ranks of $A$ and $B$ are both equal to $r$ and $| B
  |/| A | = p^r$.
\item[3.]  $N(\iota(a))=p a,\forall a\in A$.
\end{itemize}
Then
\begin{itemize}
\item[ A. ] The inclusion $\iota(A) \subset p B$ holds
  unconditionally.
\item[ B. ] Suppose that $\sexp(A) > p$. Then $\prk(A) =
  \prk(\iota(A))$(i.e. $\iota$ is \textit{rank-preserving}) and
  $\iota(A) = p B$, while $B[ p ] = \Ker(N) \subset
  \iota(A)$. Moreover, $\ord(x) = p \cdot \ord( \iota(N x) )$ for all $x \in
  B$.
\item[ C. ] If $\sexp(A) > p$ and there is a group homomorphism $T : B
  \ra B$ with $\iota(A) \subseteq \Ker(T)$ and $\nu := \iota \circ N =
  p + \binom{p}{2} T + O(T^2)$, then $\nu = \cdot p$, i.e. $\iota
  (N(x)) = p x$ for all $x \in B$.
\end{itemize}
\end{lemma}
\begin{proof}
  Since $A$ and $B$ have the same $p$-rank and $N$ is surjective, we
  know that the map $\overline{N} : B/p B \ra A/p A$ is an
  isomorphism\footnote{For finite abelian $p$-groups $X$ we denote
    $R(X) = X/pX$ by \textit{roof} of $X$ and $S(X) = X[ p ]$ is its
    \textit{socle}.}.  Therefore, the map induced by $N \iota$ on the
  roof is trivial. Hence $\overline{\iota} : A/p A \ra B/p B$ is also
  zero and thus $\iota(A) \subset p B$. This confirms the claim A.

  We now consider the map $\iota' : A/p A \ra p B/ p^2 B$ together
  with $\overline{N}$. From the hypotheses we know that $N \iota'$ is
  the multiplication by $p$ isomorphism: $\cdot p : A/ p A \ra p A/p^2
  A$, using the fact that $\sexp(A) > p$ which implies that $\prk(A) =
  \prk(\iota(A))$. It follows that $\iota'$ is an isomorphism of
  $\F_p$-vector spaces and hence $\iota : A \ra p B$ is surjective.
  From $| B | = p^r | A | = p^r | p B |$ we see that $| A | = | p B
  |$ and thus $\iota : A \ra p B$ is an isomorphism; it is in
  particular rank preserving. The cokernel of $\iota$ is an
  $\F_p$-vector space of dimension $r$.

  Taking Pontrjagin duals, the roles of $N$ and $\iota$ are
  interchanged. Hence the statement about cokernel of $\iota$ implies
  that the kernel of $N$ is also annihilated by $p$ and has order
  $p^r$; it thus coincides with $B[ p ]$ and since $\iota$ is rank
  preserving, it follows that $B[ p ] = \iota(A)[ p ] \subset
  \iota(A)$. Let now $x \in B \setminus p B$ have order $q p$ and let
  $r = \ord(\iota(N x))$. Then $N(r x) = r N(x) = 0$, so $r x \in
  \Ker(N) = B[ p ]$ and $r p x = 0$, thus $q | r$. Conversely, $q x
  \in B[ p ] = \Ker(N)$, so $\iota(q N(x)) = \iota(N(q x)) = 0$,
  implying $r | q$. Therefore $q = r = \ord(\iota(N x)) = \ord(x)/p$,
  which completes the proof of point B.

  For point C. we let $x \in B$, so $p x \in p B = \iota(A)$ and thus
  $T p x = p T x = 0$. Consequently $T x \in B[ p ] \subset \iota(A)$
  and therefore $T^2 x = 0$. From the definition of $\nu = p + Tp
  \frac{p-1}{2} + O(T^2)$ we conclude that $\nu x = p x +
  \frac{p-1}{2} T p x + O(T^2) x = p x$, which confirms the claim
  C. and completes the proof.
\end{proof}

\subsection{Kummer extensions, Property F and stabilization}
Iwasawa has proved in his classical Theorem 6 from \cite{Iw} a
property relating ramification to the first cohomology of the groups
$A(\KL_n)$.  We review here his construction, which shall be
generalized for our context; we refer the reader either to the
original paper \cite{Iw}, or the Lemmata 13.14-13.16 in \cite{Wa}.
Let $\{ \eu{P}_i \ : \ i = 1, 2, \ldots, s \} \subset \KH$ be a set of
primes above the unique primes $\wp_i \subset \K; i = 1, 2, \ldots, s$
above $p$, which ramify completely in $\KL/\K$. Since $\KH/\KL$ is
unramified, it follows that the inertia groups $I(\eu{P}_i) \subset
\Gal(\KH/\K) \cong \Gamma$ are all isomorphic to $\Z_p$ and one can
choose topological generators $\tau_i \in I(\eu{P}_i)$ which restrict
to a fixed topological generator $\tau = \tau_i \vert_{\KL} \in
\Gal(\KL/\K)$. Following Iwasawa \cite{Iw}, we let $a_i \in A$ be such
that $\tau_i = \varphi(a_i) \tau_1, i = 2, 3, \ldots, s$ and identify
a lift of $\tau$ to $\Gal(\KH/\K)$ with $\tau_1$. Let $Y =
\bigoplus_{i=2}^s \Z_p \varphi(a_i) \subset X$ and $Y_n = \omega_n X +
\nu_{n,1} Y = \nu_{n,1} ( Y + T X)$. Then Iwasawa's Theorem 6 in
\cite{Iw} states that
\begin{eqnarray}
\label{an}
A_n \cong X/Y_n, \quad Y_n = \omega_n X + \nu_{n,1} Y = \nu_{n,1} ( Y + T X ) 
\quad \forall n, 
\end{eqnarray}
and thus $a \in A$ has $a_n = 0$ iff $a \in \varphi^{-1}(Y_n)$. In
view of the Lemma \ref{sa} in the Appendix, one verifies that \rf{an}
is equivalent to
\begin{eqnarray}
\label{hiwas}
 H^1(\Gamma \vert_{\KL_n}, A_n) \cong Y_n/\omega_n X \cong Y/(Y \cap T X), 
 \end{eqnarray}
 the last isomorphism holding only for large enough $n$. If $Y \subset
 T X$, or, equivalently, $H^1(\Gamma \vert_{\KL_n}, A_n) = 0$ for all
 $n > 1$, we say that $A(\KL)$ has \textit{Property F}\footnote{The
   name recalls Furtw\"angler, who first noted this property in a
   slightly different context of class field theory.}, or simply that
 $\LK/\K$ has this property.

We retain the above facts for future reference:
\begin{lemma}
\label{ydef}
Let $\KL/\K$ be a $\Z_p$-extension in which all the primes above $p$
ramify completely, let $\Lambda$ be the associated Iwasawa algebra and
$\Gamma = \Gal(\KL/\K), X = \Gal(\KH/\KL)$. There is a finitely
generated $\Z_p$-module $Y \subset X$ such that \rf{an} and \rf{hiwas}
hold for every $n > 0$.  Moreover, $Y \not \subset T X$ iff there is
some $y \in A \setminus T A$ with $y_1 = 0$.
\end{lemma}

We shall be concerned with various phenomena of module stabilization,
for which we start by introducing
\begin{definition}
\label{stab}
Let $\KL \subset \F \subset \KH$ be a galois extension of $\K$, let
the intermediate fields be $\F_n = \F \cap \KH_n, \overline{\F}_n =
\F_n \cdot \KL$, let $X_n = \Gal(\overline{\F}_n/\KL)$ and $X =
\varprojlim_{n \geq 0} X_n$. Let $F = F(X), L = L(X), M = M(X)$ be the
modules in Definition \ref{type}, associated to $X$.

If $F' \subset F$, we say that $F'_n$ is \textit{stable}, if $F'_m
\cong F'_n \cong F'$ for all $m \geq n$.  If $L' \subset L$ then
$L'_n$ is \textit{stable} if $\prk(L'_n) = \prk(L'_m) = \prk(L')$.
Let $Y_n := H^0(\Gal(\KL_n/\K), A_n)$. We say that the $H$-part is
stable (for $m > n$) if
\[ Y_n \cong Y_{n-1} \cong Y(\F)/T X \quad \hbox{ for all $m > n$.}
\]

The smallest integer $v > 0$ such that $x_v \neq 0$ for all $x \in
A(\KL) \setminus \eu{M} A(\KL)$ is called the {\em visibility index};
more general, if $C \subset A(\KL)$ and $I \subset \Lambda$ is an
ideal, the visibility index of $C$ with respect to $I$ is $v := \min_k
\{ k \ : \ x_k \neq 0, \ \forall x \in C \setminus I C \}$.

The least integer $n_0$ for which $F, L, H$ and $M$ are stable is the
stabilization index of $X$. It will be useful to assume that the
stabilization index additionally fulfills the condition $x_0 \neq 0$
for all $x \in X \setminus I X$. Unless otherwise specified, the ideal
$I = \eu{M}$.
\end{definition}
Stabilization criteria for the module $A$ were first given by Fukuda
\cite{Fu}, in the case when $\mu(\KL) = 0$. S. Kleine has studied in
his Thesis a large spectrum of stabilization conditions in multiple
$\Lambda$-extensions.  The result we present here is a variant of the
statements proved by him.
\begin{proposition}
\label{fuk}
Let $\KL/\K$ be a $\Z_p$-extension in which the primes above $p$ are
totally ramified and let $\KL \subset \F \subset \KH$ be a galois
extension of $\K$ with group $X = \Gal(\F/\KL)$.  Then
\begin{itemize}
\item[ 1. ] If $X_n \cong X_{n+1} \neq 0$ for some $n > 0$, then $X_n \cong
  X$ and $X$ is finite.
\item[ 2. ] If $\prk(X_n) = \prk(X_{n+1}) > 0$ for some $n > 0$, then
  $\prk(X_n) = \prk(X)$ and $\mu(X) = 0$.
\item[ 3. ] Let $V_n := X_n/ T X_n$; if $V_n \cong V_{n+1} \neq 0$ then $V_n
  \cong X/T X$ and $Y_n/ T X_n \cong Y/ T X$, the $H$ - part being
  stable for $m > n$ and $X[ T ]$ finite.
\end{itemize}
\end{proposition}
\begin{proof}
  Let $Y = Y(\F) \big \vert_{\F}$ and $Y_n = \nu_{n,1}( Y + TX)$.  We
  have proved that $X_n \cong X / Y_n$ and assume without restriction
  of generality that $1$ is the least integer $n$ for first
  stabilization in both cases 1. and 2. We have the following
  commutative diagram in which $X_n \rightarrow X_{1}$ is induced by
  the map $\nu_{n,1}$ while the horizontal isomorphisms are deduced
  from the definition of $Y_n$.
\begin{eqnarray}
\label{cd}
\begin{array}{c c c}
X_n & \cong & X/\nu_{n,1} Y \\
\downarrow & & \downarrow \\
X_{1} & \cong & X/Y.
\end{array}
\end{eqnarray}
For the first point we assume $| X_{2} | = | X_{1} |$. Then $X_{2}
\rightarrow X_{1}$ is an isomorphism; therefore $\nu_{2,1} Y =
Y$. Since $\eu{M} = (p, T) \subset \Lambda$ is the unique maximal
ideal and $\nu_{2,1} \in \eu{M}$, and since $Y$ is finitely generated
over $\Lambda$, it follows from Nakayama's lemma that $Y =
0$. Consequently, $X \cong X_{1}$ and $X_n \cong X_1 \cong X$ for all
$n \geq 1$. The condition $X_2 \cong X_1$ readily implies finiteness
of the $X$, which proves the assertion 1.

Suppose now that $\prk(X_2) = \prk(X_1)$. Then $X_2/ p X_2 \cong X_1/p
X_1$ and thus $X/(\nu_{2,1} Y + p X) \cong X/( Y+ p X)$ and $\nu_{2,1}
Y + p X \cong Y + p X$. Letting $Z = (Y + p X)/pX$, we have
\[\nu_{2,1} Z = (\nu_{2,1} Y + p X)/pX = (Y + p X)/pX = Z.\] 
By Nakayama's lemma, we conclude that $Z = 0$ and $Y \subset p
X$. Therefore, 
\begin{eqnarray*}
  \prk(X_n) & = & \prk(X/\nu_{n,1} Y) = \prk(X/(\nu_{n,1} Y +
  p X)) \\
  & = & \prk(X/pX) = \zprk(X), \quad \hbox{for all} \quad n \geq 0.  
\end{eqnarray*}
By Iwasawa's formula, for $n$ sufficiently large we have
\[ | X_n | = p^{\mu p^n + \lambda n + \nu} , \] and since the rank
stabilizes, we see that $\mu(X) = 0$ and $| X_{n+1} | / | X_n | \geq
p^{\lambda}$ with equality iff $F(X) = 0$. In this case too, $\mu(X) = 0$ is a consequence of the
stabilization condition. This proves assertion 2.

Finally the stabilization of the cohomology part is analogous to point
1. We have $V_n = X_n/T X_n = X/( T X + \nu_{n,1} Y)$. Let $W_n =
\nu_{n,1} Y$ so $TX + \nu_{n,1} Y_n = TX + W_n$ while $TX +
\nu_{n+1,1} Y = T X + \nu_{n+1,n} W_n$. In exact sequences
\begin{eqnarray}
\label{wn} 
\begin{array}{c c c c c c c c c}
  0 & \ra & T X + W_n & \ra & X & \ra & X /(T X + W_n) & \ra & 0  \\
  0 & \ra & T X + W_{n+1} & \ra & X & \ra & X /(T X + W_{n+1}) & \ra & 0, \\
\end{array}
\end{eqnarray}
the isomorphism $V_{n+1} = X /(T X + W_{n+1}) \cong X /(T X + W_{n}) =
V_n$ implies that $T X + W_n \cong T X + \nu_{n+1,n} W_n$. It follows
from Nakayama's Lemma that $W_n \subset T X$; indeed, the module $Z_n
:= T X + W_n$ is finitely generated, so let $t_1, t_2, \ldots, t_r \in
T X \setminus \eu{M} T X$ be a minimal set of generators of $T X$.
Assuming that $TX \neq Z_n$ there is a minimal set of generators $w_1,
w_2, \ldots, w_j \in W_n \setminus (\eu{M} Z_n + T X)$ such that $(W_n
+ TX)/T X = \sum_j \Lambda \overline{w}_j$. But since $T X + W_n = T X
+ \nu_{n+1,n} W_n$, we deduce that $ \sum_j \Lambda \overline{w}_j =
\nu_{n+1,n} \sum_j \Lambda \overline{w}_j$ and since $\nu_{n+1,n} \in
\eu{M}$ it follows that $(T X + W_n )/TX = 0$, and $\nu_{n,1} Y
\subset T X$. A fortiori $\nu_{m,1} Y \subset T X$ and thus $Z_m = X /
T X $ for all $m > n$. It follows in particular that $X / T X$ is
finite and since $| X / T X | = | X_n / T X_n | = | X_n[ T ] | $ for
sufficiently large $n$, if follows that $H^0(\Gal(\KL_n/\K), A_m)$ is
stable for $m > n$.
\end{proof}

The strength of this Fukuda-type result is that it shows that the
first stabilization occurring within the projective sequence of galois
groups $X_n$ readily implies global stabilization.

The stabilization conditions above require no a priori knowledge about
the shape of $X$. Moreover, if $H$ is stable, then all $x \in A
\setminus \eu{M} A$ are visible. It is however not possible to
determine stabilization of $\mu$- parts from internal data, as the
following example shows:
\begin{example}
\label{mug}
Let $\K = \Q[ \sqrt{-d} ]$ be an imaginary quadratic field with
trivial $p$-part of the class field and let $\K_{\infty}$ be its
cyclotomic $\Z_p$-extension. For $n > 0$ we consider a principal prime
ideal $\eu{q} = (\gamma_n) \subset \K_n$, which is totally split over
$\Q$ and also splits in $\K[ \zeta_p ]$.  If $q \in \N$ is the
rational prime above it, then $q \equiv 1 \bmod p$ and we let $\F
\subset \Q[ \zeta_q ]$ be the subfield of degree $p$ while $\KL = \K
\cdot \F$.

Then it can be shown (see next chapter), that there is an ideal
$\eu{R} \subset \KL_n = \K_n \cdot \F$ with class $r_n = [
\eu{R}/\overline{\eu{R}} ]$ and such that $\Norm_{\KL/\K}(r_n) = 1$,
while $\Lambda r_n \cong \Lambda/(p, \omega_n)$. Assuming that
$A^-(\KL) = \Lambda r$ for a norm coherent sequence containing $r_n$,
we see that $A^-(\KL)$ has $\mu$-like growth up to level $n$, but
since $\mu(\KL) = 0$ by the Theorem of Ferrero-Washington, the
$p$-rank of $\Lambda r_m$ must stabilize for some $m > n$. This fact
cannot be detected by analyzing the sequence $r_1, r_2, \ldots, r_n$.
\end{example}
Of course, rank stabilization eventually takes place in this example,
so it can be detected by Proposition \ref{fuk}. Therefore it would be
interesting to know whether, in the case when $\mu > 0$, the rank
stabilization of some submodule can be perceived. A partial answer is
contained in Proposition \ref{fuk}, which allows choosing subfields of
the Hilbert class field -- so the question is transformed into one of
constructing an adequate subfield.

We now give some applications of the Fukuda result.  We keep the same
notation for $\K \subset \KL \subset \F \subset \KH' \subset \KH$,
with $\KH'$ being the maximal subextension of $\KH$ which splits all
the primes above $p$.

We let $\KH^{(l)} = \KH^{\varphi(M)}$, where $M = A^{\circ}$ 
and $\KH^{(t)} = \KH^{T \varphi(A)}$, the indicator
for stabilization of $H$-parts.  The galois groups are
\[ X = \Gal(\KH/\KL), \quad X^{(x)} = \Gal(\KH^{(x)}/\KL), \quad x \in
\{ l, t \}.
\]
The Proposition \ref{fuk} can be applied to these extensions in order
to establish the stabilization index $n_0$ of $\KL$.  As a direct
consequence we have
\begin{fact}
\label{lgrow}
Let $x \in X^{(l)}$; for $n$ beyond the stabilization index $n_l$ of
$X^{(l)}$ and for all $k > 0$, we have $\iota_{n,n+k}(x_n) = p^k
x_{n+k}$.
\end{fact}
\begin{proof}
  The choice of $n_l$ implies that $\prk(X^{(l)}_n) =
  \prk(X^{(l)}_{n+1}) = \prk(X^{(l)})$.  For $k = 1$, we let $B =
  X^{(l)}_{n+1}$ and $A = X^{(l)}_n$. Then $N, \iota$ are the
  restriction $N_{\K_{n+1}, \K_n}$ and the lift map. The choice of $n$
  also implies that $\sexp(A) > p$ and we let $T = \omega_n$ in Lemma
  \ref{ab}. We deduce from point C that 
  \begin{eqnarray}
  \label{bases}
  \iota{x_n} = p x_{n+1},
  \end{eqnarray}
  which is the statement for $k = 1$. The general case follows by
  induction on $k$, letting $A = X^{(l)}_{n+i}, B = X^{(l)} x_{n+i+1}$
  for $i = 0, 1, \ldots, k-1$, successively, and applying the result
  for $k = 1$ established previously. Indeed, assume that for all $j
  \leq i$ we have $\iota_{n,n+i}(x_n) = p^i x_{n+i}$. Using also the
  fact that $\iota_{n+i,n+i+1}(x_{n+i}) = p x_{n+i+1}$ which follows
  from t\rf{bases}, we find
  \begin{eqnarray*}
  \iota_{n,n+i+1} (x_n) & = & \iota_{n+i,n+i+1} \left(\iota_{n,n+i} (x_n)\right) = \iota_{n+i,n+i+1} (p^i x_{n+i}) \\
  & = & p^{i+1} x_{n+i+1}, 
  \end{eqnarray*}  
  and thus it follows by induction that $\iota_{n,n+i}(x_n) = p^i x_{n+i}$ and
  the claim follows by letting $i = k$.
\end{proof}

\subsection{Decomposition}
We let $\KL/\K$ be some $\Z_p$-extension in which all the primes above
$p$ are totally ramified and $p^B$ be the exponent of $A^{\circ}$.  We
let $n_0 > 0$ be an index such that
\begin{eqnarray}
\label{kstab}
\sexp(L_{n_0}) \geq p^4 \quad \hbox{and} \quad \prk(L_n) = 
\prk(L_{n_0}) \quad \forall n \geq n_0.
\end{eqnarray}
Note that the condition \rf{kstab} is fulfilled by all submodules $L'
\subset L$ which are spanned by elements of infinite order -- or such
ones of order at least $p^4$. We assume, without loss of generality,
that this is the case for $L$ too. We note the following
\begin{fact}
\label{down}
With the notations of this section, for all $n > n_0$ and all $x =
(x_n)_{n \in \N} \in L$, we have
\[ L_{n+2}[ p^2 ] \subset \iota_{n, n+2}(L_{n}) \quad \hbox{ $and $ } 
\quad \omega_n \cdot \omega_{n_0} (x_{n+1}) = 0 .\]
\end{fact}
\begin{proof}
  By hypothesis, we have $\sexp(L_n) \geq p^4$; since
  $\exp(\Ker(\iota_{n,n+2} : L_{n} \ra L_{n+2})) = p^2$, it follows
  that $\sexp(\iota_{n,n+2}(L_n)) \geq p^2$. The ranks are conserved,
  by hypothesis, so we conclude that $\iota_{n,n+2}(L_n) \supset
  L_{n+2}[ p^2 ]$. For arbitrary $n \geq n_0$ we have $p x_{n+1} =
  \iota_{n,n+1}(x_n)$ and thus $\omega_n x_{n+1} \in L_{n+1}[ p
  ]$. The second fact will now be proved by induction on $n$.

  For $n = n_0$ we have $\omega_n x_{n+1} \in L_{n+1}[ p ] \subset
  \iota(L_{n_0})$, and thus $\omega_n \omega_{n_0} (x_{n+1}) = 0$.
  Let now $n > n_0$ be fixed and assume that $\omega_n \omega_{n_0}
  (x_{n+1}) = 0$. Using
  \[ \omega_{n+1} = \omega_n \cdot \nu_{n+1,n} = \omega_n \cdot ( p
  u(\omega_n) + \omega_n^{p-1}), \] we conclude that
\begin{eqnarray*}
  \omega_{n_0} \omega_{n+1} (x_{n+2}) & = & \omega_{n_0} \omega_n p u(\omega_n) (x_{n+2}) + \omega_{n_0} \omega_n^p (x_{n+2}) \\ 
  & = & \iota_{n+1,n+2}\left(\omega_{n_0} \omega_n (x_{n+1}) u(\omega_n)\right) +  \omega_{n_0} \omega_n^p (x_{n+2}) 
  = \omega_{n_0} \omega_n^p (x_{n+2})
\end{eqnarray*}
where the last equality follows from the induction hypothesis. Now
$p^2 \omega_n x_{n+2} = \iota_{n,n+2}(\omega_n x_n) = 0$, hence
$\omega_n x_{n+2} \in L_{n+2}[ p^2 ] \subset \iota_{n,n+2}(L_n)$ and
thus $\omega_n^2 x_{n+2} = 0$, which completes the proof.
\end{proof}
We let $B$ be an upper bound for $v_p(\mu)$ over the $\mu$ invariants
of all $\Z_p$-extensions of $\K$, so $p^B$ is a safe upper bound for
$\exp(M(\KL)/F(\KL))$.  We let $F(T) \in \Lambda$ be the minimal
annihilator polynomial of $L(A)$ and note that $\rg{D}:= A/D$ is a
finite $\Lambda$-module.  We shall also assume, without restriction of
generality, that $n_0 = 1$ for our base field $\K$: this can be
achieved by a shift up of the base field.

Passing to decomposition, we note the following property:
\begin{lemma}
\label{pt}
Let $\KL/\K$ be a $\Z_p$-extension satisfying the condition \rf{kstab}
and let the further notations be as defined above. The modules $D, M,
L, F$ are defined with respect to $A$. If $p x \in D$ then $T^2 x, \
\omega_2 x \in D$.

Moreover, if $x \not \in D$ but $p x, T x \in D$ and $T x =
x_{\lambda} + x_{\mu}$, then $\ord(x_{\mu,1}) = p$ and $x_{\mu,1} = -
x_{\lambda,1}$.
\end{lemma}
\begin{proof}
  Let $w \in A \setminus (\eu{M} A + D)$ and suppose that $l \leq B$
  is the smallest integer such that $p^l w \in L$ and let $f_w(T)$ be
  the minimal annihilator polynomial of $p^l w$. Then $y := f_w(T) w
  \in M$ and $p^l y = 0$. There is some $0 < d \leq l \leq B$ such
  that $x := p^{d-1} w$ verifies $x \not \in D$ but $p x \in D$; note
  that $l,d$ are the orders introduced in \rf{dords}. Let $\id{X} = \{
  x \in A \setminus D \ : \ p x \in D \} \subset A$ and $\id{X}'
  \subseteq \id{X}$ be the set of those elements that arise as
  described above. Then
  \[ p^j x_{n+j} - \iota_{n,n+j}(x_n) \in f_x(T) \Lambda x_{n+j}
  \subset \Lambda y_{n+j}, \quad \forall j > 0. \] In particular
  $\iota_{n,n+l}(x_n) = p^l x_{n+l} - h_{n+l}(T) (f_x x_{n+l})$ is
  decomposed and for $n > n_0$ and $x \in A \setminus D$ such
  that $\ord(p^l x_n) > p$, we have
\begin{eqnarray}
\label{strong}
p^l x_{n+l} - \iota_{n,n+l}(x_n) = f_x(T) h_n(T) x_{n+l} \in M_{n+l}.
\end{eqnarray} 
Indeed, consider the modules $B = \Lambda x_{n+1}/(f_x \Lambda
x_{n+1})$ and $A = \Lambda x_n/ (f_x \Lambda x_n)$. Since
$\iota_{n,n+1} (x_{n}) \not \in f_x \Lambda x_{n+1}$ for $n > n_0$ --
as follows from the condition imposed on the orders -- the induced map
$\iota : A \ra B$ is rank preserving. We can thus apply the Lemma
\ref{ab}, which implies the claim \rf{strong}, and deduce under the
above hypothesis on $n$, that
\[ p^l x_{n+l} = p^{l-1} c_{n+l} = \iota_{n+1,n+l} (c_{n+1}) =
\iota_{n,n+l}(x_n) + h y_{n+l}, \quad h \in \Z_p[ T ] . \] By Fact
\ref{down} and the choice of $\K$ such that $n_0 = 1$, we have
$\omega_n T c_{n+1} = 0$.  Applying $\omega_n$ to the above identity
we find $T h \omega_n y_{n+l} = 0$. The relation \rf{an} implies that
there is some $z \in A$ such that $T^2 h \omega_n y = \omega_{n+l}
z$. In addition, we have $p^l \omega_{n+l} z = 0$. The result of Sands
of Lemma \ref{sa} yields $z \in M + A[ T ]$.  Then $\omega_n( T^2 h y
- \nu_{n,n+l} z) = 0$ implies $T^2 h y \in \nu_{n,n+l} z + A[ T ] +
F$. Since $h y \in M$, it follows that $z \in M$ and $T^2 h y \in
\nu_{n, n+l} z - \phi$, say, for some $\phi \in F$.  Reinserting this
relation in the initial identity, we find
  \begin{eqnarray}
  \label{cdeco}
  \iota_{n,n+l}( T^2 x_n + z_n) = \iota_{n+1,n+l} (T^2 c_{n+1}) - \phi_{n+l}
  \end{eqnarray} 

  Note that the right hand side is in $L$ and thus has uniformly
  bounded $p$-rank.  This leads to the following two proofs for the
  fact that \rf{cdeco} implies that $T^2 x$ must be decomposed.  For
  the first, we invoke the Lemma \ref{dlam} with respect to the
  sequence $w^{(n)} = T^2 x + z$, where the upper index stresses the
  fact that the choice of $z$ depends on $n$. Since $d_n(
  \iota_{n+1,n+l} (T^2 c_{n+1}) ) \leq \prk(L)$ for all $n$, the Lemma
  \ref{dlam} implies that there is an uniform $n_0 > 0$ such that
  $w_n^{(n)} \in \nu_{n,n_0} A + F$. But then
  \[ w_n^{(n)} = \iota_{n,n+l}( T^2 x_n + z_n) = \nu_{n,n_0} (a_n +
  f_n) \in \iota_{n_0,n}(A_{n_0}) . \] It follows in particular that
  \[ \ord(T^2 x_n + z_n) \leq p^l \ord(\iota_{n,n+l}( T^2 x_n + z_n))
  \leq p^l \exp(A_{n_0}) . \] This holds for arbitrary large $n$ and
  since $z_n \in M$ we have $\ord(T^2 x_n + z_n) = \ord(T^2 x_n)$,
  thus obtaining a contradiction if $T^2 x_n \not \in D$, case in
  which $\ord(x_n) \ra \infty$.

  The second proof uses topological facts.  If $f \in Z_p[ T ]$ is the
  minimal annihilator polynomial of $p^m x$ and thus of $c$, then we
  found that for every $n$ there is a $z = z^{(n)} \in M$ such that $f
  T^2 x_n + f z^{(n)}_n = 0$, thus
  \[ w_n := -T^2 y_n = f(T) z^{(n)}, \quad z^{(n)} \in M .\] Let $m >
  n$; by definition, we have $w_m = f(T) z_m^{(m)}$ and, since $w =
  -T^2 y$ is a norm coherent sequence, a fortiori, $w_n = f(T)
  z_n^{(m)}$. We may assume that $z^{(m)} = z^{(n)}$ and therefore,
  upon extracting subsequences from the sequence $z^{(n)}$, the
  defining condition $w_n = f(T) z^{(n)}$ is conserved.  Since $M$ is
  a Noetherian module, we may choose a minimal system of generators
  $u^{(i)} \in M \setminus \eu{M} M, i = 1, 2, \ldots, s$ and let
  $z^{(n)} = \sum_{i=1}^s c^{(n)}_i u^{(i)}, c^{(n)}_i \in \Lambda$,
  where the representation is not unique. We obtain thus a sequence
  $(C_n)_{n \in \N}$ with $C_n = \left(c^{(n)}_i\right)_{i=1}^s \in
  \Lambda^s$. In the $\eu{M}$-adic product topology, $\Lambda^s$ is a
  compact space. Letting $p^{B} M = 0$, we see that we may choose
  $c^{(n)}_i \in \Z_p[ T ]$ as polynomials with degree
  $\deg(c^{(n)}_i) \leq \deg \omega_n$ and coefficients of valuation
  at most $B$. There is a converging subsequence $C_{n_i}$. After
  eventual renumeration, we may thus assume that the sequence $C_n$ is
  convergent. Let $C = (c_i)_{i=1}^s = \lim_n C_n$ and let for all $n$
  the polynomial $\omega_{n,B} \in \Z[ T ]$ have coefficients in $\{
  0, 1, \ldots, p^{B}-1\}$ and verify $\omega_{n,B} \equiv
  \omega_n \bmod p^{B}$. Note that the polynomials $c^{(n)}_i$ are
  all defined modulo $\omega_{n,B}$, while $c_i$ is defined modulo
  $p^{B}$.  After eventually extracting a new subsequence, we may
  assume that the $C_n$ are such that
\begin{eqnarray}
\label{conv}
c^{(n)}_i - c_i \in \omega_{n, B} \Lambda, \quad \quad 
\hbox{ for all $n > 0$ and $i = 1, 2, \ldots, s$} .  
\end{eqnarray}
Let $z = \lim_n z^{(n)} = \sum_i c_i u^{(i)}$. From $w_n = f(T)
z^{(n)}$ we deduce that $w_n = \sum_i f(T) c_i^{(n)} u^{(i)}_n$ and
since $c_i^{(n)} \equiv c_i \bmod \omega_{n,B}$ it follows that $w_n
= \sum_i f(T) c_i u^{(i)}_n = f(T) z_n$. We have thus proved that $w =
f(T) z$ for some $z \in M$ and thus $f(T) ( z + T^2 x ) = 0$, hence $z
\in T^2 x + L$, which proves that $T^2 x \in D$ as claimed. Moreover,
$\omega_2 = T (p u(T) + T^{p-2})$ and since $p x$ and $T^2 x \in D$ it
follows also that $\omega_2 x \in D$, which completes the (second)
proof of the first statement.

Suppose now that $T x = x_{\mu} + x_{\lambda}$ and $p x_{\mu} =
x_{\mu,1} = 0 $. Let $\id{N} - p = p s f_2 (T) + s^{p-1}, f_2(T) \in
\Lambda^{\times}$, and note that, in stable growth, $\omega_n
x_{\lambda,n+1} \in L_{n+1}[ p ]$ and thus $\omega_n^2 x_{\lambda,n+1}
= p \omega_n x_{\lambda, n+1} = 0$. Since $p x_{\mu} = 0$, we have
\begin{eqnarray*}
  -p x_{n+1} + \iota_{n,n+1} x_n & = & 
  p \nu_{1,n} f_2(\omega_n) (x_{\lambda, n+1} + x_{\mu, n+1}) + 
  \nu_{1,n} \omega_n^{p-2} (x_{\lambda, n+1} + x_{\mu, n+1}) \\
  & = & T^{D_n-1} x_{\mu,n+1}, \quad D_n = \deg(\omega_{n+1}-\omega_n) = \deg(\nu_{n,n+1}).
\end{eqnarray*}
Writing $r_{n+1}$ for the right hand side in the above identities, we
consequently obtain $\omega_n r_{n+1} = \nu_{n+1,1} x_{\mu,n+1} =
x_{\mu,1} = 0$. By Lemma \ref{ydef}, there is thus some $z \in A$ such
that $\omega_n r = \omega_n (\nu_{n+1,n} - p) x = \omega_{n+1} z$, so
$r = \nu_{n,n+1}(z_{n+1})$ and consequently $\iota_{n,n+1}(x_{n+1} -
z_{n+1}) = -p x_{n+1} \in L_n$. However, we have seen that for $x \not
\in D$ and $y \in D$, the distance $d_{n}(x, y) \ra \infty$; in
particular the distance on the left hand side of the last identity
will diverge, while the right hand side has upper bounded distance,
since $p x \in L$.  This contradiction implies that for $x \in A
\setminus D$ such that $p x, T x \in D$, we must have $x_{\mu,1} \neq
0$. Since $x_{\mu,1} + x_{\lambda,1} = T x_1 = 0$, it follows that
$x_{\mu,1} = - x_{\lambda,1}$.  Since we assume that the growth of $A$
is stable from the ground field, we have $\ord(x_{\lambda, 2}) = p
\ord(x_{\lambda,1}) = \ord(T x_2) = p$, thus $\ord(x_{\lambda,1}) =
\ord(x_{\mu,1}) = p$, which completes the proof.
\end{proof}
For individual $\Z_p$-extensions, we have:
\begin{proposition}
\label{deco}
Let $\KL/\K$ be a $\Z_p$-extension and $A, \Lambda$ be associated to
$\KL$ as usual. If $p^B$ is the exponent of ${A}^{\circ}$, then
$\eu{M}^{2B} A \subset D(A)$ and $\omega_B A \subset D(A)$.
\end{proposition}
\begin{proof}
  For $x \in A$ we let $k = \ord_D(x) = \min \{ j : p^j x \in D \}$ be
  the \textit{decomposition order} of $x$. The proof will follow by
  induction on $k$, on base of the Lemma \ref{pt}

  For $k = 1$, it is a direct consequence of the lemma, since
  $\eu{M}^2 = (p, pT, T^2)$. Assume that the statement holds for all
  $x \in A$ with $\ord_L(x) < k$ and note that $\eu{M}^{2k} = (p^2, p
  T, T^2 ) \eu{M}^{2(k-1)}$.  Since we assumed that $p^k x \in D$, it
  follows that $p x$ has order $k-1$ and by induction hypothesis, we
  have $p \eu{M}^{2(k-1)} x \subset D$. For arbitrary $w \in
  \eu{M}^{2(k-1)} x$ we have thus $p w x \in D$ and the Lemma \ref{pt}
  implies that $T^2 w x \in D$. The choice of $w$ being free, it
  follows that $T^2 \eu{M}^{2(k-1)} x \subset D$ too, hence
  $\eu{M}^{2k} x \subset D$. This holds for all $k$, and letting $k =
  m$ we conclude that $\eu{M}^{2 m} A \subset D$, which completes the
  proof. The fact $\omega_m A \subset D$ follows from Lemma \ref{pt}
  by induction too, the proof being similar.
\end{proof}

As a consequence, 
\begin{corollary}
\label{nm}
Let $p^B$ be the exponent of $M$ and suppose that $n' > n_0 + B$ with
$n_0$ the stabilization index of $\KL$; if we shift the base field
according to $\K_1 = \K_{n'}$ and redefine $\Lambda$ accordingly, then
$T A \subset D$.
\end{corollary}
\begin{proof}
  Let us write $\Lambda^{(0)}, T^{(0)}, \omega^{(0)}$, etc for the
  Iwasawa algebra and its elements, defined with respect to the
  initial base field $\K^{(0)}$, say. We have then $T =
  \omega^{(0)}_{n'}$. A simple computation shows that $\omega_n \in
  \eu{M}^n$ for all $n$, so then $\omega^{(0)}_{n'} \in
  (\eu{M}^{(0)})^{2B}$ and the claim follows from the Proposition
  \ref{deco}.
\end{proof}

The results above are indicative for what can be achieved in full
generality.  In our context, we shall need the following specific
application for CM fields:
\begin{lemma}
\label{decob}
Let $\K'$ be a CM galois extension of $\Q$ containing the \nth{p}
roots of unity and let $\KL'/\K'$ be a $\Z_p$ CM extension. The
modules $A, D, L, F$ are defined with respect to this extension and we
consider $x \in A^-$ such that $p x = c + v \in D^-$ with $c \in L^-,
v \in M^-$, such that $\omega_n T c_{n+1} = 0$ for all $n \geq 0$.
Then $T x \in D^-$
\end{lemma}
\begin{proof}
  The proof is identical to the one of Lemma \ref{pt}. Note the
  difference in premise: here we cannot make a global statement on the
  stability of $L^-$ for all $n >0$, but we do have sufficient
  information about the decomposition of $p x$, so that the proof can
  be completed like in the proof of the Lemme \ref{pt}, the details
  being left to the reader.
\end{proof}

\subsection{On CM $\Z_p$-extensions of number fields}
In this section we gather several properties of CM $\Z_p$-extensions
which are the base for our approach; recall that the occurrence of CM
$\Z_p$-extensions different from the cyclotomic one, is
\textit{equivalent} to the failing of Leopoldt's conjecture for CM
fields $\K$. We let $\K$ be some galois CM number field for which the
Leopoldt conjecture fails and let $\K_{\infty}$ be its cyclotomic
$\Z_p$-extension. We let $\M$ be the compositum of all the
$\Z_p$-extensions of $\K$, let $\M_0^+$ be the compositum of all the
$\Z_p$-extensions of $\K^+$ and $\M^+ = \K \cdot \M_0^+ \subset \M$.

The radicals of $\M^+$ as a Kummer extension of $\K_{\infty}$ are
intimately related to the failure of Leopoldt's conjecture and the
$T^*$-part of the class group, by the following \textit{folklore}
result, which holds in the cyclotomic $\Z_p$-extension of a field:
\begin{proposition}
\label{fantom}
Let $\K$ be a CM field which contains the \nth{p} roots of unity and
$A(\K) = \varprojlim_n A(\K_n)$ be defined with respect to the
cyclotomic $\Z_p$-extension. Then
\[ \zprk(A^-[ T^* ]) = \id{D}(\K) , \] and in particular Leopoldt's
conjecture fails for $\K$ iff $A^-[ T^* ] \neq 0$. Moreover
\begin{eqnarray}
\label{radfan}
 \M^+ \subseteq \K_{\infty}[ (A^-( T^* ))\pinf ].
\end{eqnarray}
In particular, for every CM $\Z_p$-extension $\KL/\K$ there is a class
$a \in A^-( T^*) \setminus T^* A^-(T^*)$ such that $\KL \cdot
\K_{\infty} = \K_{\infty}[ a\pinf ]$.
\end{proposition}
The proof of the proposition is given in the Appendix.

For the cyclotomic $\Z_p$-extension, it is known that
$A^-(\K_{\infty})$ has no finite $p$-torsion submodule. In the case of
non-cyclotomic CM $\Z_p$-extensions, this fact is almost true, namely:
\begin{lemma}
\label{almin}
Let $\K$ be a CM extension containing the \nth{p} roots of unity and
$\KL/\K$ be a CM $\Z_p$-extension with $\KL \cap \K_{\infty} = \K_N, N > 1$ 
and write $\KL_N := \K_N; \ [ \KL_{N+n} : \K_N ] = p^n$. If $\mu_{p^N}
\subset \KL$ but $\mu_{p^{N+1}} \not \subset \KL$ then the finite
torsion submodule $C^- := F(A^-) \subset A^-$ is a cyclic group of order $p^N$.
If $a = (a_n)_{n \in \N}$ is a generator of $C^-$, then
\begin{eqnarray}
\label{defl}
\KL_{N + m} = \KL_{N+m-n}[ a_{N+m-n}^{1/p^n} ] \quad \hbox{  for all $m, N \geq n$},
\end{eqnarray}
with the root of a class defined like in Remark 1.
Suppose that $T^* = T - p^k, 1 \leq k \leq N$ is the
Iwasawa involution and assume that $N$ is chosen such that
\begin{eqnarray}
\label{nbed}
p^{2k } a \neq 0 \quad \hbox{ for all $a \in A_{N-1}^-( T^* ) \subset p A_N^-( T^* )$.}
\end{eqnarray}
Then $T^* C^- = 0$.
\end{lemma}
\begin{proof}
  Let $c \in C^- \setminus p C^-$ generate a direct term of order $q
  := p^j; \ j \leq N$ in the abelian $p$-group $C^-$.  Let $\eu{C} \in
  c_m, m > N$ be a prime ideal. Since $c$ is a finite torsion element,
  it follows that $\iota_{m,\infty}(c_m) = 0$, so we may assume that
  $l \geq j$ is the least integer such that $\iota_{m, m+l}(c_m) =
  0$. In the sequel we show that we must in fact have $l = j$.  Let
  $\eu{C}^q = (d)$ and $\iota_{m,m+l}(\eu{C}) = (\delta)$. Since $\KL$
  is CM, we conclude from Kronecker's unit Theorem, after eventually
  modifying $\delta$ by some root of unity, that
  \[ \delta/\overline{\delta} = (d/\overline{d})^q.  \] We can thus
  apply Kummer theory in the abelian cyclic extension
  $\KL_{l+m}/\KL_m$. The minimality of $l$ implies that
  $\delta/\overline{\delta} \not \in ((\KL_{m+l-1})^{\times})^q$. By
  an inductive repetition of the argument, it follows that
  $\ord(\iota_{m,m+l-j}(c_m)) = \ord(c_m)$ and $\iota_{m,m+l-j}(\Z
  c_m) \cong \Z c_m$; thus
\begin{eqnarray}
\label{tors}
\KL_{m+l} = \K_{m+j-l}[ (d/\overline{d})^{1/q} ] = \KL_{m+l-k}[ c_m^{1/q} ]. 
 \end{eqnarray}
 This implies $j = l$; according to point B in the Remark \ref{shift},
 there is a class $a_m \in A_m^-$ of order $\ord(a_m) = p^j$, such
 that $\KL_{m+j} = \KL_m[ a_m^{1/p^j} ]$. One verifies by using the
 same computations as above, that $\iota_{m,m+N}(a_m) = 0$. Taking a
 norm coherent sequence $a = (a_i)_{i \in \N}$ through $a_m$, we see
 that $\ord(a) \geq p^N$. Together with the inequality $j \leq N$, we
 conclude that $\exp(C^-) = p^N$. Moreover, we have shown that there
 is a sequence $a = (a_m)_{m \in \N} \in C^-$ with $\ord(a) = p^N$ and
 $\KL_{m+N} = \KL_m[ a_m^{1/p^N} ]$ for all (sufficiently large) $m$.

 We claim that $C^-$ is $C^-$ is $\Z_p$-cyclic, so $C^- = \Z a$. Since
 $C^-$ is a

%
 finite abelian $p$-group, we have $\prk(C^-) = \dim_{\F_p}(C^-[ p ])
 = \dim_{\F_p}(C^-/p C^-)$. We show that $C[ p ] \cong \F_p$, and thus
 $\prk(C^-) = 1$. Let $m$ be fixed and $c_m \in C_m[ p ]$ be a class,
 the primes of which become principal in $\KL_{m+1}$. If $\eu{C} \in
 c_m$ is a prime, $(\gamma) = \eu{C}^p$ and $(\delta) =
 \iota_{m,m+1}(\eu{C})$, then we showed that we may assume
 $\gamma^{1-\jmath} = \delta^{p(1-\jmath)}$. On the other hand, we
 have shown above that the sequence $a$ can be chosen such that
 $\KL_{m+N} = \KL_m[ a_m^{1/p^N} ]$ and in particular $\KL_{m+1} =
 \KL_m[ a_m^{1/p} ]$.  Concretely, let $(\gamma) = \eu{A}_m^{p^N},
 \eu{A}_m \in a_m$. Then Kummer theory implies that there is an
 integer $v$, coprime to $p$, such that
 \[ \frac{\alpha}{\overline{\alpha}} =
 \left(\frac{\gamma}{\overline{\gamma}} \right)^v \cdot w^p, \quad w
 \in \KL^{\times}_m. \] Then $(\eu{C}/\eu{A}_m^{v p^{N-1}})^{1-\jmath}
 = (w)$ and, in terms of classes, we conclude that $c_m = a_m^{v
   p^{n-1}}$. Since $c_m$ was chosen arbitrarily, it follows that
 $a_m$ generates $C_m^-$ and thus $C^- = \Z a$ is a cyclic $p$-group
 of order $p^N$.

 Finally, $T^* C^- = 0$ follows from the Proposition
 \ref{fantom}. Indeed, we have $\KL_{2 N} = \KL_N[ a_N^{1/p^N} ]$ and
 the proposition implies that $a_N \in A_N[ T^* ]$, with $T^* = T -
 p^k$ for some fixed $k < N$, which depends on the choice of
 $\K$. Since the annihilator polynomial $f_a(T)$ is linear, say $f-a =
 T - v p^j$ and $T^* a_N = f_a(T) a_N$, it follows that $a_N (q - v
 p^j) = 0$; if $f_a \neq T^*$, then $q a_N = 0$. This is inconsistent
 with the choice of $N$, which completes the proof.
\end{proof} 

As a consequence:
\begin{lemma}
\label{deco-}
Suppose that $\KL/\K$ is a CM $\Z_p$-extension in which all the primes
above $p$ are totally ramified and $\K$ is chosen such that the
conditions in Proposition \ref{deco} hold.  Then $T A^- \subset D^-$
and there is a decomposition $T A^- = L_t + M_t$ with $L_t \cap M_t
\subseteq C^-$; in particular, $T^* L_t \cap T^* M_t = 0$.
\end{lemma}
\begin{proof}
  The primes above $p$ are totally ramified and the base is chosen
  such that the Proposition \ref{deco} holds, thus we can apply
  complex conjugation, obtaining $T A^- \subset D^-$ and $L_t \cap M_t
  = C^-$ by definition of the $\mu$ and the $\lambda$-parts. The final
  claim follows from $T^* C^- = 0$. Let $w \in T^* L_t^- \cap T^*
  M_t^-$ be given by $w = T^* x \in T^* L_t^-, w = T^* y \in T^*
  M_t^-$, with $x \in L_t^-, y \in M_t^-$. Since obviously $w \in
  C^-$, we have $T^* w = (T^*)^2 y = 0$, so $y \in M_t^- \cap L_t^- =
  C^-$ and thus $T^* y = 0$.
\end{proof} 

\section{The main Theorem}
We start by fixing the context of fields in which we perform the
proof.  Suppose that $\K_{-3}$ is a CM number field in which the
Leopoldt conjecture is false. As mentioned above, we use negative
indices for a sequence of field extensions which preserve the CM
property and have a positive Leopoldt defect, while enjoying an
increasing sequence of useful properties. Eventually, $\K = \K_1
\supset \K_{-3}$ will be a ground field for which we are going to
prove that $\id{D}(\K) = 0$, thus confirming the claim of Theorem
\ref{main}.  First let $\K_{-2} = \K_{-3}^{(n)}[ \zeta_p ]$ be the
normal closure of $\K_{-3}$ to which we adjoined the \nth{p} roots of
unity. Next we choose a \textit{\em small} complex abelian extension
$\rg{k}$ such that $A^-(\rg{k}) \neq 1$ and $\rg{k} \cap \K_{-2} =
\Q$; this extension will be chosen in order to satisfy certain useful
properties which are provided in Lemma \ref{kex}.  We let $\K_{-1} =
\rg{k} \cdot \K_{-2}$. We shall wish to apply the decomposition
results above, so we let $B$ be the constant granted by the Theorem of
Babaicev and Monsky and $p^B$ also annihilates the $\Z_p$-torsion of
$\K_{\infty}$.

Let $n_0 > 0$ be the stabilization index
of $L(\K_{\infty}/\K_{-1})$ and let $n' \geq n_0 + 2B$ be such that for
all coherent sequences
\[ x = (x_{-1}, x_0, \ldots, x_{n'}, \ldots ) \in A^-(\K_{\infty})
\setminus \left(\eu{M} A^-(\K_{\infty}) + M^-(\K_{\infty})\right) \]
we have $\ord(p^B x_{n'}) \geq p^2$.

We define $\K \subset \K_{\infty}$ such that $[ \K : \K_{-1} ] =
p^{n'+1}$ and $x_{n'} \in A(\K)$. From now on $\K$ is our base
field. We note that the constant $B$ is not modified by replacing
$\K_{-1}$ with $\K$. The shift of the base field $\K$ induces also a
shift of $\rg{k}$ which describe in more detail below.

\begin{lemma}
\label{kex}
There is an imaginary abelian extension $\rg{k}/\Q$ and a class
sequence $h = (h_n)_{n \in \N} \in A^-(\rg{k})$ such that the module
$H := \Lambda h$ has rank $\prk(H) p(p-1)$ and finite index in
$A^-(\rg{k})$.  Moreover, if $T^k h \in p A^-(\rg{k})$, then $k \geq
p$ and if $D$ is any integer, $\rg{k}$ can be chosen such that the
primes dividing $D$ are unramified in $\rg{k}$.
\end{lemma}
The proof uses results that will be developed in the next sections and
it is provided in \S 3.3. We let $D = \disc(\K_{-2})$ and use this
discriminant in the definition of a field $\rg{k}'$, using Lemma
\ref{kex}; with this, we let $\K_{-1} = \rg{k}' \cdot \K_{-2}$, as
mentioned above. Then $\K_{-1} \subset \K \subset \K_{\infty}$ is
constructed as in the previous section and we let $\rg{k} = \K \cap
\rg{k}'_{\infty}$. We also define
\begin{eqnarray}
\label{kdef}
\Delta  = \Gal(\K/\Q), \quad \Delta_0 = \Gal(\rg{k}/\Q), \quad \Delta_1 = \Gal(\rg{K}/\Q),
\end{eqnarray}
so $\Delta = \Delta_0 \times \Delta_1$. We then fix a sequence
\begin{eqnarray}
\label{alphadef}
 \alpha = (\alpha_n)_{n \in \N} \in A^-(\K_{\infty}) \quad \hbox{with} \quad \Norm_{\K_{\infty}/\rg{k}}(\alpha) = h.
\end{eqnarray}
The following construction puts in evidence CM $\Z_p$-extensions whose
existence is equivalent to the failure of the Leopoldt conjecture for
$\K$, and in which we shall use the sequence \rf{alphadef}.
\begin{lemma}
\label{lat}
Notations being like above, for arbitrary $n > 0$ there are infinitely
many prime ideals $\eu{q} \in \alpha_n$ which are totally split in
$\K_n/\Q$ and such that the decomposition group $D(\eu{q}) \subset
\Gal(\M^+/\K)$ fixes an extension $\M_q^+ \subset \M^+$ with $\K_n
\subset \M^+_q$ and $\zprk(\Gal(\M_q^+/\K_n)) > 0$.

In particular, there is a CM $\Z_p$-extension $\KL/\K$ which contains
$\K_n$ and in which $\eu{q}$ is totally split. 

\end{lemma}
\begin{proof}
  Let $\eu{q} \in \alpha_n$ be a prime ideal which is totally split in
  $\K/\Q$ and coprime to $p$. By a classical application of
  Tchebotarew's Theorem, there are infinitely many such primes. Since
  $\eu{q}$ is coprime with $p$ and all the primes that ramify in
  $\M^+/\K$ lay above $p$, it follows that $D(\eu{q}) \cong
  \Z_p$. Indeed, $\eu{q}$ is totally inert in $\K_{\infty}/\K_{n'}$
  for some $n' > n$, so we have $\zprk(D(\eu{q})) \geq 1$; since
  $\Q_q$ has only one (unramified) $\Z_p$-extension, it follows that
  $\eprk(D(\eu{q})) = 1$. But $\Gal(\M^+/\K) \cong
  \Z_p^{\id{D}(\K)+1}$ has no finite subgroups and thus $D(\eu{q})
  \cong \Z_p$, as claimed. = Moreover, $\K_n \subset \M_q^+ :=
  {\M^+}^{D(\eu{q})}$ since we chose $\eu{q}$ to be completely split
  in $\K_n$.  We have $\zprk(\Gal(\M_q^+/\K)) = \id{D}(\K) > 0$ and
  there is in particular some CM $\Z_p$-extension $\KL \subset
  \M^+_q$. By definition, $\KL \cap \K_{\infty} \supseteq \K_n$ and
  the prime $\eu{q}$ is totally split in $\KL$.

\end{proof}

\subsection{Thaine shift and the main coherent sequences}
We let $\K$ be a galois CM extension constructed as above, so we
assume in addition that stabilization occurs from the first level in
the following sense
\begin{itemize}
\item[ A. ] For all $x \in A^-(\K_{\infty}) \setminus (\eu{M}
  A^-(\K_{\infty}) + M^-(\K_{\infty}))$ we have $\ord(p^B x_1) > p$
  where $p^B$ is an exponent for the $\mu$-part of all the
  $\Z_p$-extensions of $\K$ (the existence of which follows from
  Theorem \ref{mubound}).
\item[ B. ] The shift equation $p x_n = \iota_{n-1,n} (x_{n-1})$ holds
  for all $n \in \N$ and $ x \in L(A^-)$ with $x_{n-1} \neq 0$.
\item[ C. ] We have $\mu_{p^k} \subset \K$ but $\mu_{p^{k+1}} \not
  \subset \K$ and $\K = \K_1 = \ldots \K_k \subsetneq \K_{k+1}$
\end{itemize}

Let $N = 2M > 0$ be an integer to be determined below and let $\KL/\K$
be some $\Z_p$-extension with $\KL \cap \K_{\infty} \supseteq \K_N$,
for instance the one constructed before; the case $\KL =\K_{\infty}$
is in particular allowed too.  We define the following {\em Thaine
  shift extensions}: let $\eu{r} \in \alpha_n, n \leq N$ be some
totally split prime which is inert in $\K_{n+1}/\K_n$ and $r \equiv 1
\bmod p$ be the rational prime above $\eu{r}$; we let $\F \subset \Q[
\zeta_r ]$ be the subfield of degree $p$ over $\Q$. Since $\eu{r}$ is
totally split in $\K$ while $r$ is ramified in $\F$, we have $\K \cap
\F = \Q$. We let $F = \Gal(\F/\Q)$ be generated by $\nu = \nu_r$, let
$s = s_r = \nu_r -1$ and write $\id{N}_a = \Norm_{\F/\Q}$ for the
arithmetic norm, while the algebraic norm is
\begin{eqnarray}
\label{anorm}
\id{N} & = & \sum_{i=0}^{p-1} \nu^i = p u(s) + s^{p-1} = 
p + s f(s), \\ & & \nonumber f \in \Z_p[ X ], \quad u \in (\Z_p[ s ])^{\times}. 
\end{eqnarray}

We define $\K^{(r)} = \K \cdot \F$ and $\KL^{(r)}_n = \KL_n \cdot \F,
\ \KL^{(r)} = \KL \cdot \F$. The galois groups $\Delta : =
\Gal(\K/\Q), \Gamma := \Gal(\KL/\K)$ commute with $F$ and thus
$\Gal(\K^{(r)}/\Q) = F \times \Delta, \Gal(\KL^{(r)}/\K^{(r)}) =
\Gamma, \Gal(\KL^{(r)}/\K) = F \times \Gamma$. 
\begin{itemize}
\item[ I ] In the case when $n < N$ and $\eu{r}$ is inert in $\KL$, 
we say that $\KL^{(r)}/\KL$ is an {\em inert Thaine shift}.  
\item[ S. ] If $\eu{r} = \eu{q}$ is totally split in $\KL$ we speak of a 
{\em split Thaine shift}.
\end{itemize}

The split case is applied for the proof of Theorem \ref{main}, while
the inert shift is used in the construction of the auxiliary extension
$\rg{k}$ in Lemma \ref{kex}. In the split case, $\KL \cap \K_{\infty}
=: \K_N = \KL_N$ for some $N = 2 M > 0$, where $\KL = \KL^{(q)}$.  The
prime $\eu{q} \subset \K_N$ is totally split in $\KL/\Q$ and we let $(
\eu{q}_m )_{m \in \N}$ with $\eu{q}_m \subset \KL_m$ be a norm
coherent sequence of primes above $\eu{q}$. Moreover, we assume that
$\eu{q}$ is inert in $\K_{N+1}/\K_N$ and if $q$ is the rational prime
below $\eu{q}$, then $q \equiv 1 \bmod p^N$.  In the split Thaine
shifts of our context, we shall assume that $q$ verifies this
conditions. We shall denote by $\xi \in \mu_{p^N}$ a primitive root of
unity, so $\xi \in \KL'_n$ generates the group of $p$-roots of unity,
for all $n$. As a consequence, we have
\begin{fact}
\label{hasse}
Let $\KL'/\KL$ be a Thaine split shift and the related conditions and
notations be like above.  Then $\xi \not \in
\id{N}((\KL'_n)^{\times})$ for all $n > N$. Moreover, if $d = [ \K_N^+
: \Q ]$, then
\begin{eqnarray}
\label{normdefect}
\big \vert (\KL_n^{\times})^- / \id{N}\left( ({\KL'_n}^{\times})^- \right)\big \vert = p^{p^{n-N} d} . 
\end{eqnarray}
\end{fact}
\begin{proof}
  This is a direct consequence of the Hasse norm principle. Let indeed
  $\eu{r} \subset \KL_n$ be any prime above $q$, thus any of the
  primes that ramify in $\KL'_n/\KL_n$; the claim follows by showing
  that $\xi$ in not in the local norm image. Since $\eu{r}$ is totally
  split, the residue field is $\F_q$ and the $p$-Sylow of the
  multiplicative subgroup has size $| (\F_q^{\times})_p | =
  p^N$. Local class field theory implies that the norm image has index
  $p$, so it is a cyclic subgroup $C_{p^{N-1}}$ and can therefore not
  contain the full image of the \nth{p^N} roots of unity. A fortiori,
  $\xi \not \in \id{N}((\KL'_n)^{\times})$ for all $n > N$, which
  completes the proof of the first statement. Note that the number of
  pairs of conjugate primes above $q$ in $\KL_n$ is $R := p^{p^{n-N}
    d}$ and the Hasse Norm Principle implies that the size of the norm
  defect $\KL_n^{\times} / \id{N}\left( {\KL'_n}^{\times} \right)$ is
  equal to the product of the local norm defects at each of these
  ramified primes -- which are the only primes that ramify in
  $\KL'_n/\KL_n$. Since we have seen that the local norm defects are
  groups of order $p$, the claim \rf{normdefect} follows by taking
  minus parts.
\end{proof}

The primes $\eu{q}_m$ are totally ramified in $\KL'_m$ and we let
$\eu{Q}_m \subset \KL'_m$ be the ramified prime above $\eu{q}_m$; in
particular $\eu{Q}_0$ is the prime of $\K_N$ above $\eu{q}$. This
leads to the definition of two sequences which play a crucial role in
our proof: we let
\begin{eqnarray}
\label{abdef}
a_{m} & = & [ \eu{q}_m^{1-\jmath} ], \quad \hbox{for $m \geq N$ and} \quad 
a_m = \Norm_{N,m}(a_m), \quad m < N, \nonumber \\
b_{m} & = & [ \eu{Q}_m^{1-\jmath} ], \quad \hbox{for $m \geq N$ and} \quad b_m = \Norm_{N,m}(b_m), \quad m < N, \\
a & = & (a_m)_{m \in \N} \in A^-(\KL), \quad b = (b_m)_{m \in \N} \in A^-(\KL'). \nonumber
\end{eqnarray}
It follows from the definition that $h_n = \Norm_{K_n, \rg{k}_n},
(a_n)$ for $n \leq N$ and $b = p a$, as sequences and thus at all
levels, due to ramification. We let $C' = F(\KL') \subset A^-(\KL')$
be the maximal finite submodule. The following lemma indicates the
choice of $N$:
\begin{lemma}
\label{nmm}
Notations being the ones above, one can choose $N = 2 M$ such that
there are $a_{\lambda} \in L^-(\KL), a_{\mu} \in M^-(\KL)$ with $a =
a_{\lambda} + \omega_M a_{\mu}$ and $\omega_M \cdot M(A^-) \cap F =
0$.  Moreover, $T b \in D^-(\KL')$.
\end{lemma}
\begin{proof}

  Let $f \in \Z_p[ T ]$ be the annihilator polynomial of $\alpha$
  defined in \rf{alphadef}, so $f a_N = 0$. The base
  field was chosen such that $T a \in D^-(\KL)$, so 
let $T a = a'_{\lambda} + a'_{\mu}$. Since $f a_N =
  0$, an application of Iwasawa's Theorem 6 implies that 
there is an $x \in A^-(\KL)$ for which we have
\[ f T a_N = \omega_N x = f (T) \cdot (a'_{\lambda} + a'_{\mu}) =
\nu_{N,1} (x_{\lambda} + x_{\mu}). \] By comparing parts - and using
the fact that the intersection $L^- \cap M^- \subset C^-$ is
annihilated by $T^*$, we obtain $f(T) T^* a'_{\mu} = \nu_{N,1} T^*
x_{\mu}$. Euclidean division yields $\nu_{N,1} = g(T) f + r(T)$, so
that for sufficiently large $N$ we have $r(T) \equiv 0 \bmod p^B$. For
such $N$, $\nu_{N,1} x_{\mu} = f(T) g(T) x_{\mu}$ and thus $f ( T ) (
a'_{\mu} - g(T) x_{\mu} ) = 0$. Since $\mu$ - parts are not
annihilated by distinguished polynomials, it follows that $a'_{\mu} =
g(T) x_{\mu}$. We still have to show that we may choose $ N = 2M $
such $g(T) \equiv \omega_M h(T) \bmod p^B$, which is equivalent to
$\nu_{2M,1} \equiv h(T) f(T) \omega_M \bmod p^B$. Once again,
Euclidean division yields $\nu_{2 M,1} = Q(T) \cdot (f(T) \omega_M) +
R_M(T)$. For all roots $\xi \in \overline{\Q}_p$ of $f(T) \omega_M$ we
have
\[ \nu_{2 M, 1}(\xi) = \frac{(\xi+1)^{p^{2M}}-1}{(\xi+1)^{p^k}-1} =
R_M(\xi), \] It suffices thus to take $M$ large enough, so that the
global $p$-adic valuation is $v_p(R_M(\xi)) > B$ for all zeroes of $f$
and $\omega_M$. Since for $f$, the zeroes are fixed, the problem is
solved by taking $M$ sufficiently large. For zeroes of $\omega_M$ we
use the development $\nu_{2M,1} = \nu_{M,1} \cdot \left(p^M +
  O(\omega_M)\right)$. It follows that for sufficiently large $N =
2M$, we have $\nu_{2 M,1} \equiv Q_1(T) f(T) \omega_M \bmod p^B$ and
we may also assume that the quotient $Q_1$ has free coefficient
$Q_1(0) \equiv 0 \bmod p^B$, so $Q_1(T) \equiv T Q_2(T) \bmod
p^B$. Letting $y = Q_2(T) x_{\mu}$ we have found that $a'_{\mu} =
\omega_M T y$. Then $T (a - \omega_M y) = a'_{\lambda} \in L$, so $a -
\omega_M y \in L$ too, thus $a = \omega_M y + w, w \in L$, which
yields the claimed decomposition.

Finally, since $p b = a$, we may apply Lemma \ref{pt} and deduce that
$T b \in D^-(\KL')$, after eventually shifting the base field up by
one level. Since $F = \Ker(\psi : M(A^-) \ra \id{E}(M))$ for any
pseudoisomorphism $\psi$, we can choose $M$ sufficiently large, such
that $\omega_M \psi$ is injective, which implies $\omega_M M(A^-) \cap
F = 0$.
\end{proof}

There are thus $b_{\lambda} \in L^-(\KL'), b_{\mu} \in M^-(\KL')$ with
$T b = b_{\lambda} + b_{\mu}$. From $s b = 0$ we also have $s
b_{\lambda} = -s b_{\mu} \in L^- \cap M^- = C' \cap \Ker(\id{N}) = C'[
p ] = C[ p ] \subset A^-(\KL)$, so $s^2 b_{\lambda} = 0$; also, $p s
b_{\lambda} = T s b_{\lambda} = 0$. Consequently, $\id{N}(b_{\lambda})
= p u(s) b_{\lambda} + s^{p-1} b_{\lambda} = p b_{\lambda}$ and
$\id{N}(b_{\mu}) = p b_{\mu}$. Since $p (b_{\lambda} + b_{\mu}) = p T
b = T a = T (a_{\lambda} + \omega_M a_{\mu})$, it follows by comparing
parts that $T a_{\lambda} - p b_{\lambda} = p b_{\mu} - T \omega_M
a_{\mu} = \gamma \in C^-$.  Upon multiplication by $T^*$ we obtain $p
T^* b_{\lambda} = T T^* a_{\lambda}$. The same proof yields $p T^*
b_{\mu} = T T^* a_{\mu}$.

We have the following defining relations:
\begin{eqnarray}
\label{betadef}
p b  & =  & \id{N}(b) = a, \quad \quad s b  =  0 \\
T b & = & b_{\lambda} + b_{\mu}, \quad  T T^* a_{\lambda} 
= p T^* b_{\lambda} , \quad p T^* b_{\mu} = \omega_M T T^* a'_{\mu}. \nonumber
\end{eqnarray}

Since we have shown above that $s (T b_{\lambda}) = s (T b_{\mu}) =
0$, it will be important to investigate in more detail the group
cohomology $H^0(F, A^-(\KL'))$.  This is done in the next section, in
which we also show that $b_{\mu} \neq 0$.

 \subsection{Cohomology and the Hasse obstruction module}
 In this section we investigate the Tate cohomology groups in inert
 and in split Thaine shifts. Let $\K$ be a fixed galois CM extension
 containing the \nth{p} roots of unity, and $\KL/\K$ be a CM
 $\Z_p$-extension with $\KL \cap \K_{\infty} = \K_N$, and we let $\KL'
 = \KL \cdot \F$ be a Thaine shift. Thus the extension tower can be
 the one defined in the previous section, but the above are the only
 prerequisites that we shall need in this section. The
 Tate-cohomologies in Thaine shifts are governed by the Hasse Norm
 Principle and similar properties which are ingredients of the proof
 of Chevalley's Theorem, also called the ambig class formula
 \cite{La}, Chapter 13, Lemma 4.1. These facts allow comprehensive
 descriptions of the groups. Here we only focus on the facts that are
 directly needed in our subsequent proof.

 We consider first the case when $\KL'/\KL$ is \textit{a split Thaine
   shift}, and let like above $\eu{q}_n \subset \KL_n$ build a norm
 coherent sequence of split primes above $\eu{q} \in a_N$; let $a_n :=
 [ \eu{q}_n^{1-\jmath} ]$ and $\eu{Q}_n \subset \KL'_n$ be the
 ramified ideals above $\eu{q}_n$, while $b_n = [ \eu{Q}_n^{1-\jmath}
 ]$. Since $\eu{q}$ is assumed to be totally split above $\Q$, there
 are $D_N := [ \K_N : \Q ]/2$ pairs of complex conjugates primes above
 $q$ in $\K_N$, with $(q) = \Z \cap \eu{q}$. We assume that $r'$ of
 these are totally split in $\KL$ and let $\tau_i \in \Delta_N :=
 \Gal(\K_N/\Q)$ with $\tau_1 = 1$ and $i \leq r'$ be automorphisms
 such that $\id{R} = \{ \eu{q}^{(i)} := \tau_i \eu{q} \ : \ i = 1, 2,
 \ldots, r' \}$ be these totally split primes. We denote by
 $(\eu{q}_n^{(i)})_{ n \in \N} $ some fixed norm coherent sequences of
 primes above $\eu{q}^{(i)}$ and let the class sequences $a^{(i)},
 b^{(i)}$ be defined with respect to these sequence, by analogy to the
 way $a, b$ were defined with respect to $\eu{q}_n$.  We may write,
 with some abuse of language, $\eu{q}^{(i)} = \tau_i \eu{q}, a^{(i)} =
 \tau_i a, b^{(i)} = \tau_i b$.  Let $f = f_a(T) \in \Z_p[ T ]$ be the
 minimal annihilator polynomial of $a$ and note that $f$ also
 annihilates $\Lambda b/(\Lambda b \cap M)$.

Therefore $z := f_a(T) T^* \in M^-$ is such
that $\Lambda z \cap F = 0$ and thus $\Lambda z \cong \Lambda/p^e$ for
some fixed exponent $e$.  If $z^{(i)} = \tau_i z = f_a(T) T^* \tau_i
b$, then $\Lambda z^{(i)} \cong \Lambda/p^{e(i)}$ for some $e(i) >
0$. Let $r \leq r'$ and the ordering of the $z^{(i)}$ be such that
\[ M_B := \sum_{i=1}^{r'} \Lambda z^{(i)} = \sum_{i=1}^r \Lambda
z^{(i)}, \] and $r$ be minimal with this property; i.e. $\{ z^{(i)}, i
= 1, 2, \ldots, r \}$ is a minimal spanning set for the
$\Lambda$-module $M_B$. We claim that $M_B = \bigoplus_{i=1}^r \Lambda
z^{(i)}$.  Indeed, let $\psi : M_B \ra \id{E}(M_B)$ be a
pseudoisomorphism. Since the kernel is a finite $\Lambda$-module,
while $M_B \subset T^* M$ contains no finite submodules, it follows
that $\psi$ is injective, so $M_B$ is a direct sum. The claim now
follows by induction. We show that every span of $m$ terms in $M_B$ is
a direct sum. This is true for $m = 1$. Suppose that the claim holds
for all $n < m$ but there is, after eventual reordering, a sum $x =
\sum_{i=1}^m c_i(T) z^{(i)} = 0$. We assume that $v_p(c_m)$ is
minimal, so the identity can be rewritten as $p^a( v_m(T) z^{(m)} + y
) = 0$, with $v_m \in \Z_p[ T ]$ a distinguished polynomial and $y \in
\sum_{i=1}^{m-1} \Lambda b^{(i)}$.  Let $\psi(b^{(i)}) = E_i$ and let
$M'_B = \bigoplus_{i=1}^{m-1} \Lambda E_i \subset \psi(M_B)$. We
assume that $E_{m}, \ldots, E_{r} \in \id{E}(M_B) \setminus M'_B$ are
chosen in order to complete a $\Lambda$-base of $\id{E}(M_B)$.  We let
$\psi' = \psi \big \vert_{\bigoplus_{i=m}^r \Lambda E_i}$ and $w = p^a
z^{(m)}$. Then $\psi'(v_m(T) w) = 0$ so injectivity implies that $p^a
z^{(m)}$ is a finite torsion element. But $M_B \cap C^- = 0$ so $p^a
z^{(m)} = 0$, which confirms that $x = 0$ and completes the proof by
induction.

We note that if $(\eu{q}'_n)^{(i)}$ is some other sequence above
$\eu{q}^{(i)}$ and $(b'_n)^{(i)}$ are the respective classes, then
$b^{(i)} - {b'}^{(i)} \in \omega_N A^-(\KL')$ and in particular, both
sequences have the same image in $H^0(F, A^-(\KL'))/(T^{p^N})$.  We
assume that $r \leq r'$ is maximal such that the classes $\tau_i b$
are $\Lambda$-independent.  In particular, $r$ does not depend on the
choice of $b^{(i)}$.

We have shown:
\begin{lemma}
\label{dirt}
Let $\id{R} = \{ \eu{q}^{(i)} := \tau_i \eu{q} \ : \ i = 1, 2, \ldots,
r' \}$ be the set of all conjugates of $\eu{q}$ which are totally
split in $\KL$ and let $\tau_i b$ be defined as above, while $\tau_i z
= f_a(T) T^* \tau_i b$, and $f_a(T)$ is the minimal annihilator of
$a_{\lambda} \in A^-(\KL)$. Then there is a constant $r \leq r'$ such
that
\[ M_B := \sum_{i=1}^r \Lambda \tau_i z = \sum_{i=1}^{r'} \Lambda \tau_i z, \]
and $r$ is minimal with that property and the sum is direct.
\end{lemma}

With these notations we also have
\begin{lemma}
\label{h0}
Let $\KL/\K$ be a CM $\Z_p$-extension with $\KL \cap \K_{\infty} =
\K_N$ and which allows a split Thaine shift $\KL'/\KL$. Let the
notations introduced above for primes and their classes hold; then
\begin{eqnarray}
\label{kers}
\Ker\left( s \ : \ A^-(\KL') \ra A^-(\KL') \right) = \iota(A^-(\KL)) + \sum_{j=1}^r \Lambda \tau_i(b). 
\end{eqnarray}
\end{lemma}
\begin{proof}
  Let for $n > 0$ the group $T_n \subset A^-(\KL'_n)$ be the
  $\Lambda$-module spanned by classes of ramified ideals, thus $T_n
  \supseteq [ b_n^{(i)}; i = 1, 2, \ldots, r ]_{\Z}$. Since finitely
  many primes above $\eu{q}$ are inert, while the only ramified primes
  in $\KL'/\KL$ are the primes above $q$, it follows that equality
  holds for sufficiently large $n$.

  Let $n > N$ be arbitrary and $c \in A_n^-(\KL'_n)$ have non trivial
  image $\overline{c} \in H^0(F, A_n^-)$. If $\eu{C} \in c$ then
  $\eu{C}^s = (\gamma)$ and $(\id{N}(\gamma)) = (1)$. There is thus a
  unit $\delta \in \KL_n$ with $\id{N}(\gamma) = \delta$, as algebraic
  numbers. Let $\xi \in \mu_{p^N}$ generate the $p$-roots of unity in
  $\KL'_N$. The Kronecker unit theorem implies
  \[ \id{N}\left(\frac{\gamma}{\overline{\gamma}}\right) = \xi^e,
  \quad e \in \Z. \] By Fact \ref{hasse}, $\xi \not \in \id{N}(\KL')$
  and therefore $e = p e' \equiv 0 \bmod p$.  Then $\xi^e \in
  \id{N}((\KL'_n)^{\times})$ and by applying Hilbert's Theorem 90, it
  follows after eventually modifying $\gamma$ by some root of unity,
  that $\gamma^{1-\jmath} = x^s$ for some $x \in \KL'_n$. Consequently
  $(\eu{C}^{1-\jmath}/(x))^s = (1)$ and the class $a^2 = \left[
    \eu{C}^{1-\jmath}/(x) \right]$ contains an ambig ideal; since $a
  \not \in \iota(A^-(\KL_n))$, we must even have $a^2 \in T_n$, which
  implies the claim.
\end{proof}

As a consequence, we have the following stronger result:
\begin{proposition}
\label{h0split}
Let $\KL' = \KL \cdot \F$ be a split Thaine shift with $F =
\Gal(\F/\Q) = \lan \nu \ran = \lan s+1 \ran$ and let $\tau_i \eu{q}, i
= 1, 2, \ldots, r$ be the conjugates of $\eu{q}$ that are totally
split in $\KL/\K$, while $\tau_i b_n, \tau_i b$ are the classes
defined previously in this context.  Then
\[ H^0(F, A^-(\KL')) \cong \bigoplus_{i=1}^r \F_p[[ T ]] \tau_i
\overline{b} \] is a free $\F_p[[ T ]]$ module of rank $r > 0$; here
the $\tau_i \overline{b}$ are the images of $\tau_i b$ in $H^0(F,
A^-(\KL'))$. We have $\mu(\KL') \geq r$.
\end{proposition}
\begin{proof}
  We already know from the previous lemma that $H^0(F, A^-(\KL'))
  \cong M_B / M_B \cap \iota(A^-(\KL'))$ with $M_B = \sum_{i=1}^r
  \Lambda \tau_i(b)$. By ramification, we have $p b = \iota(a)$ and
  thus $p H^0(F, A^-(\KL')) = 0$, which makes $H^0$ into an $\F_p[[ T
  ]]$ module. We have shown in Lemma \ref{dirt} that $M_B$ is free of
  $\F_p[[ T ]]$-rank $r$, which implies $\mu(\KL') \geq r > 0$.

  We are left to prove that $\prk(\Lambda b) = \infty$.  The relation
  \rf{normdefect} implies that $| H^1(F, A^-(\KL_n')) | \geq
  p^{p^{n-N} d}$.  Indeed, let $R = p^{p^{n-N} d}$ and $\eu{r},
  \overline{\eu{r}} \subset \KL_n$ one of the $R$ pairs of complex
  conjugate primes above $q$. Let $g \in \F^{\times}_q$ generate the
  $p$-Sylow subgroup and let
  \[ x_{\eu{r}} \in \id{O}(\KL_n) \quad \hbox{with} \quad x_{\eu{r}}
  \cong \begin{cases}
    g \bmod \eu{r} & \hbox{,} \\
    1/g \bmod \overline{\eu{r}} & \hbox{and} \\
    1 \bmod \eu{r}'
  \end{cases}, \] for all primes $\eu{r}' \supset q \id{O}(\KL_n) \cap
  (q)$ and $(\eu{r}', \eu{r} \overline{\eu{r}} ) = (1)$. By applying
  the Tchebotarew Theorem to the $q$-ray class field, we deduce that
  there is a principal prime ideal $(\rho)$ with $\rho \equiv
  x_{\eu{r}} \bmod q \id{O}(\KL_n)$, which is totally split in
  $\KL'_n/\Q$. We let $\eu{R} \subset \KL'_n$ be the prime above it
  and $r = [ \eu{R}^{1-\jmath} ]$; thus $r$ is not
  $p$-principal. Otherwise, $r$ is annihilated by some power $t$ with
  $(t, p) = 1$ and we may assume that $\eu{R}^{t(1-\jmath)} =
  (\gamma)$.  But then $\id{N}(\gamma/\overline{\gamma}) \in \mu_{p^N}
  \cdot \rho$.  Let $P = \prod_{i=1}^{ 2 R }
  (\F_q^{\times})^{(q-1)/p^N} \subset \id{O}(\KL_n)/q \id{O}(\KL_n)$
  be the product of the $p$ groups in the $q$-ideles of $\KL_n$. The
  Chinese Remainder Theorem implies that $| P^-/(P^-)^p | = R$ and for
  each residue class in $x \in P^-/(P^-)^p$ we may find $\rho, \eu{R}$
  as above, such that $\rho$ has image $x$ in $P^-/(P^-)^p$ and
  consequently $\eu{R}$ is not $p$-principal. This implies our
  claim. The groups $A^-(\KL'_n)$ are finite, so we deduce from the
  structure of $H^0( F, A^-(\KL'_n))$ that
\[ p^{s \prk(\Lambda b_n)} = | H^0( F, A^-(\KL'_n)) | = | H^1( F, A^-(\KL'_n)) | \geq p^{d p^{n-N}}, 
\]
hence $s \cdot \prk(\Lambda b_n) \geq d p^{n - N}$ and thus
$\prk(\Lambda b_n) \ra \infty$, which implies $b \not \in L^-$ and
completes the proof.
\end{proof}

\subsection{Completion of the auxiliary constructions}
As mentioned previously, the case of inert Thaine shifts will be used
in the construction of the auxiliary extension $\rg{k}$. In this case
we are particularly interested in the group $H^1$, as reflected in
\begin{lemma}
\label{inert}
Let $K$ be an imaginary quadratic extension of $\Q$ with $A^-(K) =
0$. Let $L = K^{(q)} F$ be an inert Thaine shift, with $\eu{q} \subset
K_j$ a totally split, principal prime ideal, that is inert in
$K_{\infty}/K_j$ and $F \subset \Q[ \zeta_q ]$ the cyclic subfield of
degree $p$ over $\Q$. Then $\lambda^-(L) = \varphi(p^j)$ and there is
an element $h \in A^-(L) \setminus \eu{M} A^-(L)$ such that $[ A^-(L)
: \Lambda h ] < \infty$. The module $H(p) := \sum_{i=0}^{p^{j-1}-1}
\Z_p T^i h$ is a $\Z_p$-pure submodule of $A^-(L)$: if $p^c x \in
H(p)$, then $x \in H(p)$.

Finally let $U = \omega_l(T), l \geq 1$ and $\Lambda' = \Z_p[[ U ]]$;
considering the $\Z_p$-extension $L_{\infty}/L_l$ and the induced
module $B = A^-(L_{\infty}/L_l)$ as a $\Lambda'$-module, then $H'(p)
:= \sum_{i=0}^{p^{j-1}-1} \Z_p U^i h \subset B$ is also a $\Z_p$-pure
module.
\end{lemma}
\begin{proof}
  We start by choosing $K = \Q[ \sqrt{ -d } ]$, an imaginary quadratic
  extension with $p \nmid h(K)$ and $\lchooses{-d}{p} = -1$. Such a
  field can be found since the analytic class number formula and
  bounds yield $h(K) < \sqrt{d} < p$ for $d < c p^2$, a range in which
  a discriminant can be found, which also verifies the quadratic
  reciprocity condition, requiring that $p$ is inert in $K$.

  Let $\eu{q} \in K_j$ be a principal prime which is totally split
  over $\Q$ and inert in $K_{j+1}/K_j$, let $q$ be the rational prime
  below it. We assume that $q \equiv 1 \bmod p^j$, which can be
  achieved by an application of Tchebotarew: consider the compositum
  $H[ \zeta_p ]$ with $H/K$ the maximal abelian unramified
  extension. Then $\eu{q}$ should be totally split in $H[ \zeta_p ]$,
  the existence being granted by Tchebotarew.  Let $\F \subset \Q[
  \zeta_q ]$ the subfield of degree $p$, so $\F \cap K = \Q$ since $q$
  is unramified in $K$ but totally ramified in $\F$. Let $L = \K \cdot
  \F$. Then, an application of Kida's formula implies that $A^-(L) =
  (p-1) p^{j-1}$: indeed, there are $p^{j-1}$ pairs of complex
  conjugate primes that ramify in $L_j/K_j$ and since they are inert
  in $L_{\infty}/L_j$, there are as many pairs of ramified primes in
  $L_{\infty}/K_{\infty}$. Since $K$ contains no \nth{p} roots of
  unity and the ramification index is $e = p$ for all ramified primes
  while $\mu(L_{\infty}) = 0$, the Kida formula yields
  \[ \lambda^-(L) = [ L : K ] \lambda^-(K) + (e-1) \cdot [ K_j : \Q
  ]/2 = 0 + (p-1) p^{j-1}, \] as claimed.

  We let $F = \Gal(L/K), \nu' \in F$ be a generator and $t = \nu'-1$
  and estimate $H^1(F, A^-(L_n))$ in a similar way to the one used for
  the split case above. Let $g \in \F_q^{\times}$ be a generator of
  this group and $\gamma = g^{(q-1)/p^j}$. Since $\eu{q}$ is totally
  split in $K_j$ we have
\begin{eqnarray}
 \label{crt}
 \quad \quad  \id{O}(K_j)/(q \id{O}(K_j)) & \cong & \prod_{i=1}^{p^{j-1}} \left( 
   \id{O}(K_j)/\tau^i \eu{q} \id{O}(K_j) \times \id{O}(K_j)/\tau^i \overline{\eu{q}} \id{O}(K_j) \right) \\
 & \cong & \prod_{i=1}^{p^{j-1}} (\F_q \times \F_q). \nonumber 
\end{eqnarray}
If $x \in (K_j^{\times})^{1-\jmath}$ and $x \equiv h \bmod \tau^i
\eu{q}$ then complex conjugation induces $x \equiv 1/h \bmod \tau^i
\overline{\eu{q}}$, for any (fixed) value of $i$. Let $w \in
(K_j^{\times})^{1-\jmath}$ be such that
\[ w \equiv \begin{cases}
             \gamma \bmod \eu{q} & \\
  1/\gamma \bmod \overline{\eu{q}} & \hbox{ and } \\
1 \bmod \eu{r} & \hbox{ for all other primes  $\eu{r} \subset K_j$ above $q$}.
            \end{cases}
\]
Let $\pi : (\id{O}(K_j))_q \ra \id{O}(K_j)/(q)$ be the natural
projection of the algebraic semilocalization at the primes above $q$
and $R = \Z_p[ T ] \pi(w)$; we note that $R$ is the $p$-Sylow subgroup
of the minus part of the multiplicative group in \rf{crt}. All the
primes above $q$ are ramified in $L/K$ and these are the only ramified
primes.  Since $K_j$ contains no \nth{p} roots of unity, the Hasse
Norm Principle implies that
\[ N_D := \left(K_j^{\times}/\id{N}(L_j^{\times})\right)^{1-\jmath}
\cong R/R^p.\] We claim that there is a group isomorphism $\psi :
H^1(F,A^-(L_j)) \ra R/R^p$. For this we note first that for $x \in
A^-(L_j)$ we necessarily have $\id{N}(x) = 0$, by choice of $K$. Thus
$H^1(F,A^-(L_j)) = A^-(L_j)/(t A^-(L_j))$, while $\id{N} = p u(t) +
t^{p-1}$ readily implies that $p$ annihilates $H^1(F,A^-(L_j))$.  Let
now $x \in A^-(L_j)$ with non trivial image $\overline{x} \in
H^1(F,A^-(L_j))$ and let $\eu{R} \in x$ be a totally split prime. Then
$(\rho) = \id{N}(\eu{R})($ must be a principal prime and we claim that
$\pi (\rho^{1-\jmath}) \not \in R^p$. Otherwise, $\rho^{1-\jmath} \in
\id{N}(L_j^{\times})$ and we may assume that $\rho^{1-\jmath} =
\id{N}(y^{1-\jmath})$, so in terms of ideals
$\id{N}(\eu{R}/(y))^{1-\jmath} = (1)$ and thus
$(\eu{R}/(y))^{1-\jmath} = \eu{D}^s$, for some ideal $\eu{D} \subset
L_j$. This implies that $x \in \A^-(L_j)^s$ and thus $\overline{x} =
1$, which contradicts the choice of $\eu{R}$. We define $\psi(x) =
\pi'(\rho)$ where $\pi' : (\id{O}(K_j))_q \ra R/R^p$ is the
composition of $\pi$ with the natural map $R \ra R/R^p$.  A direct
verification establishes that $\psi$ is a well defined map of $\F_p[ T
]$-modules; we leave these details to the reader and show that $\psi$
is a bijection.  We have shown that $\psi(x) = 1$ iff $x = 1$, so
$\psi$ is injective; it is also surjective. For this we consider some
principal prime $(\rho) \subset K_j$ which is totally split in
$L_j/\Q$, with $r := \pi'(\rho/\overline{\rho}) \in R/R^p$ and $r \neq
1$. Such a prime can be determined with Tchebotarew's Theorem, by
considering the $q$-ray class field $H_q \supset K_j$, which also
contains $L_j$. If $\eu{R}$ is the split prime above $(\rho)$, then
$\eu{R}/\overline{\eu{R}}$ cannot be principal, since otherwise
$\rho/\overline{\rho} \in \id{N}(L_j^{\times})$ in contradiction with
$r \neq 1$.  Letting $x = [ \eu{R} ] \in A^-(L_j)$ we see by
construction that $\psi(x) = r$ and thus $\psi$ is surjective.

Since the module $R/R^p$ is a $\F_p[ T ]$-cyclic of order $p^{j-1}$,
it follows that $H^1(F, A^-(L_j))$ is $\F_p[ T ]$ cyclic of order
$p^{j-1}$ too. We let $h = \psi^{-1}(\pi'(w))$, where $\psi'(w)$
generates $R/R^p$ as an $\F_p[ T ]$ module. We note that $R, w, h,
\psi$ all depend on $j$ and, for all $n \geq j$ there is a module
$R_n$ and a bijection of $\F_p[ T ]$-modules $\psi_n : H^1(F,
A^-(L_n)) \ra R_n/R_n^p$, which is constructed in a similar way as
above. One may choose a norm coherent sequence $h = (h_n)_{n \in \N}$
such that $h_n$ generates $H^1(F, A^-(L_n))$ as an $\F_p[ T ]$-module.

We claim that $[ A^-(L) : \Lambda h ] < \infty$.  Since $\id{N} h = 0$
it follows that $-p h = t^{p-1} u^{-1}(t) h$ so the $t$-rank of
$\Lambda[ t ] h$ is at most $p-1$; here the $t$-rank is the rank of
$\Lambda[ t ] h/\Lambda h$. Let $f \in \Lambda$ be the minimal
annihilator polynomial of $h$ and $g \in \Lambda$ be the one of $s^k
h$ for some $k < p-1$. We claim that $f = g$; we have indeed $f (s^k
h) = s^k ( f h ) = 0$ so $g \ | \ f$. On the other hand, $s^k g h = 0$
implies that
\[ - p g h = (u^{-1} s^{p-1-k}) (s^k g h) = 0, \quad \hbox{ hence}
\quad g h = 0, \] so $g \ | \ f$ too, and thus $f = g$ as claimed.

Assuming now that $[ A^-(L) : \Lambda h ] = \infty$, it follows that
there is some $k < p-1$ such that $\Z_p t^k h \cap \Lambda h = 0$.
Suppose there are $c \geq 0, g(T) \in \Lambda$ with $p^c t h = g(T)
h$, and by iteration, $p^{(p-1)c} t^{p-1} = g(T)^{p-1} h = p u(t) h$,
so $\Z_p t^k h \cap \Lambda h \neq 0$ for all $k > 0$ and thus $[
A^-(L) : \Lambda h ] < \infty$. It remains that $\Z_p t h \cap \Lambda
h = 0$ and thus $A^-(L) = \bigoplus_{i=0}^{p-2} t^i \Lambda h$, so in
particular $\Lambda h \cap t A^-(L) = 0$.  However, from $\id{N} h =
0$ we deduce that $p h = -t^{p-1} u^{-1}(t) \in \Lambda h \cap t
A^-(L)$.  This is a contradiction which implies that this case cannot
occur and thus $[ A^-(L) : \Lambda h ] < \infty$.

Let $G(T) \in \Z_p[ T ]$ be a distinguished polynomial with $G(T) h =
t v(T) h \in t \Lambda h$. We show that $\deg(h) \geq p^{j-1}$;
indeed, an iteration yields $G(T)^i h = (v(T) t)^i h, \ 0 \leq i <
p$. Inserting this relation in $\id{N} h = 0$ we obtain $( G(T)^{p-1}
+ O(p) ) h = 0$, so Weierstrass preparation implies that $h$ has an
annihilator $H(T) = G(T)^{p-1} + O(p)$ of degree $\deg(H) = (p-1)
\deg(G) \geq (p-1) p^{j-1}$.  Therefore $\deg(G) \geq p^{j-1}$, as
claimed. The same argument implies that $\prk(\Lambda h / t \Lambda h)
= p^{j-1}$.  We have $A^-(L) = \sum_{i=0}^{p-2} t^i \Lambda h$, so for
arbitrary $z \in A^-(L)$ there are polynomials $z_i(T) \in \Z_p[ T ]$
with $\deg(z_i) < p^{j-1}$ such that $z = \sum_{i=0}^{p-2} z_i(T) t^i
h$.
 
We show that the module $H(p) = \sum_{i=0}^{p^{j-1}-1} \Z_p T^i h
\subset A^-(L)$ is $\Z_p$-pure.  Indeed, consider $x \in A^-(L)$ such
that $p^c x \in H(p)$ for some $c > 0$ and let
\[ p^c x = g(T) h = p^c \sum_{i=0}^{p-2} t^i x_i(T) h,\] with
$\deg(g), \deg(x_i) < p^j$. By separating terms, we obtain $(g(T) -
p^c x_0(T)) h \in t \Lambda h$.  Since $\deg(g(T) - p^c x_0(T)) <
p^{j-1}$ it follows from the previous remarks, that $g(T) = p^c
x_0(T)$, so $p^c x \in p^c \Lambda h$, thus $x \in \Lambda h$, as
claimed.

Finally, we prove that the shifted module $H'(p)$ is also pure. First
note that $[ B : \Lambda' h ] < \infty$, the proof being identical
with the one above, after replacing $T$ by $U$ and $\Lambda$ by
$\Lambda'$. We obtain a decomposition
\[ B = \sum_{i=0}^{p-2} t^i \Lambda' h ,\] and the proof that $H'(p)
:= \sum_{i=0}^{p^{j-1}-1} \Z_p U^i h$ is $\Z_p$-pure follows the same
pattern as the one for $H(p)$.
\end{proof}

We now relate the construction above to some given CM field. Let
$\B/\Q$ be the $\Z_p$-extension of $\Q$ and let the intermediate fields
be numbered by $\B_1 = \Q$ and $[ \B_{n+1} : \B_n ] = p$. For a number
field $\M$ we let $l = l(\M) \geq 1$ be defined by $\M \cap \B =
\B_l$. With this we note the following
\begin{fact}
\label{dk}
Let $\M$ be a galois CM number field containing the \nth{p} roots of
unity with $l = l(\M)$ and $d = \disc(\M)$. The fields $K, L$ in the
Lemma \ref{inert} can be chosen such that the fields $L, \M$ are linearly 
disjoint and $(\disc(L), p d) = 1$.
\end{fact}
\begin{proof}
  If $K = \Q[ \sqrt{D} ]$, then
  $\rad(\disc(L)) = \rad(D q)$; it suffices thus to choose $D$ and $q$
  such that $(\rad(D) q, p d) = 1$. Since the discriminants are
  coprime, it also follows that $L$ and $\M$ are linearly disjoint, we
  set $\M' = L \cdot \M$. 
\end{proof}

We can complete the construction of the auxiliary fields in our
proof. Let $\K_{-3}, \K_{-2}$ be defined like at the beginning of this
chapter. With $d = \disc(\K_{-2})$ we construct $K, L$ as in Lemma
\ref{inert} and in Fact \ref{dk}, where we choose $j = 2$, so $\prk(A^-(L))
= p(p-1)$.  We let $\K_{-1}$ be the smallest field in the cyclotomic
$\Z_p$-extension of the compositum $\K_{-2} \cdot L$, in which
conditions A., B.  and C. at the beginning of \S 3.1 hold. Then
$\rg{k} := L_{\infty} \cap \K_{-1}$ and $k = l(\K_{-1}) =
l(\rg{k})$. We choose $N = 2 M$ like in Lemma \ref{nmm} and let
$\eu{q} \subset \K_N$ be a totally split prime, such that $\eu{q} \cap
\rg{k}_N \in h_N$ and $\KL/\K$ is a $\Z_p$-extension with $\K_N = \KL
\cap \K_{\infty}$ and $\eu{q}$ totally split in $\KL/\Q$ and inert in 
$\K_{\infty}/\K_N$. The field
$\F \subset \Q[ \zeta_q ]$ is herewith well defined, and we let $\K' =
\K \cdot \F, \KL' = \KL \cdot \F, \rg{k}' = \rg{k} \cdot \F$; we let
$h' \in A^-(L')$ be a sequence with $N_{L'/L}(h') = h$ and such that
$\eu{Q} \cap \rg{k}'_N \in h'_N$. Note that $h' \not \in \eu{M}
A^-(\rg{k}'_{\infty})$.

If the sequences $a, b$ are defined on base of $\eu{q}$, as described
after Lemma \ref{nmm}, then $b \not \in \eu{M} A^-(\KL')$. Indeed,
$N_{\KL',\rg{k}'}(b) = h'_N$ and the claim follows from the respective
claim on $h'_N$; a fortiori, $a \not \in \eu{M} A^-(\KL)$ and $a_N
\not \in A^-(\K_N)$. As a consequence we have:

\begin{corollary}
\label{lowpart}
In the construction defined above, assume that there exists a
$\gamma_M \in A^-(\K_M), \gamma_M \neq 0$, and a distinguished
polynomial $g(T) \in \Z_p[ T ]$ such that $p \gamma_M = g(T) a_M \not
\in \Ker(N : A^-(\K_M) \ra A^-(\rg{k}_M))$.  Then either $\gamma_M \in
\Lambda a_M$ or $\deg(g(T)) > p-1$.
\end{corollary}
\begin{proof}
  Let $w_M = \Norm_{\K_M/\rg{k}_M}(\gamma_M) \in A^-(L_M)$ (we use
  here the notation $L_M = \rg{k}_M$ in conformity with the notation
  used in the general treatment above). We have shown in Lemma
  \ref{inert} that $H'(p) := \sum_{i=0}^{p-1} \Z_p T^i h$ is a
  $\Z_p$-pure module. Due to the choice of $\K$, we shall have $l(\K)
  = l(L) geq 2$ in this case, hence the notation $H'(p)$ -- however,
  we use the variables $\tau; T = \tau - 1$ for a generator of the
  galois group in the shifted extension $L_{\infty}/L$.

  We chose $M >> l(\K)$ so in particular, the rank of $A^-(L_n)$ is
  stable for $n > 1$ and for $x \in A^-(L) \setminus \eu{M} A^-(L)$ we
  have $x_1 \neq 1$ and $\ord(x_M) \geq p^{M}$.  Let $w \in A^-(L)$ be
  an arbitrary sequence which coincides with $w_M$ in $L_M$. Such a
  sequence is unique modulo $\omega_M A^-(L)$. We have shown that
  there is a decomposition $w = \sum_{i=0}^{p-2} W_i(T) t^i h$. Taking
  norms in the identity $p \gamma_M = g(T) a_M$ we obtain
\[ p w_M = \sum_{i=0}^{p-2} p W_i(T) t^i h_M = g(T) h_M \quad \Rightarrow \quad
 (g(T) - p W_0(T)) h_M \in t \Lambda h_M .
\]
Let $V(t, T) := \sum_{i=0}^{p-3} p W_{i-1}(T) t^i$ and define 
\[ z := (g(T) - p W_0(T)) h, \quad y := t V(t,T) h \in A^-(L) .\] We
have $z_M = y_M$, so by Lemma \ref{ab} and the choice of $L$ it
follows that there is some $Z \in A^-(L)$ with $z = y + \omega_M
Z$. Let $Z = \sum Z_i(T) t^i h$, so $ (g(T) - p W_0(T) - \omega_M(T))
h \in t \Lambda h$. By choice of $M$ we have $\deg(\omega_M) > p^2$,
say. If $g(T)$ is not $p$-divisible, the Weierstrass preparation
theorem implies that there is a polynomial $\tilde{g} \in \Z_p[ T ]$
of degree $\deg(\tilde{g}) = \deg(g)$ and a unit $u(T) \in
\Lambda^{\times}$ such that $u \cdot \tilde{g} h \in t \Lambda h$.
Therefore $\tilde{g} h \in t u^{-1} \Lambda h = t \Lambda h$ and the
fact that $H'(p)$ is $\Z_p$-pure implies that $\deg(g) \geq
p$. Otherwise, $g(T) = p g_1(T)$ and thus $p (w_M - g_1(T) h_M) = 0$.
In this case, we find like previously that $w = g_1(T) h +
O(\omega_M)$ and thus $w \in \Lambda h$, which completes the proof.
\end{proof}

\subsection{Proof of the main Theorem}
We assume that $\K_{-3}$ is a CM field in which the Leopoldt
conjecture fails, and use the auxiliary constructions completed in the
previous section, in order to obtain $\K, \KL, \KL'; a, b, \eu{q}$,
etc.  Recall that, as noted above \rf{betadef},
\[ s b_{\lambda} = -s b_{\mu} \in \Ker(\id{N}) \cap (L^- \cap M^-) =
\Ker(\id{N}) \cap C^- = C^-[ p ] . \] Therefore $T b_{\lambda} , T
b_{\mu} \in \Ker(s)$. The Theorem \ref{main} is proved as follows:
\begin{proof}
  We show first that $b_{\lambda} \in \iota(A^-(\KL))$. By Proposition
  \ref{h0split}, $ \id{H} := H^0(F, A^-(\KL'))$ is a free $\F_p[[ T
  ]]$ module of rank $r \geq 1$ and we write $\overline{ \cdot } :
  A^-(\KL') \ra \id{H}$ for the natural projection. Since $T
  b_{\lambda} \in \Ker(s)$ and $T^{\deg(f) +1} b_{\lambda} \in p
  \Lambda b_{\lambda} \subset \iota(A^-(\KL))$, it follows that
  $T^{\deg(f) +1} \overline{b}_{\lambda} = 0$. Thus
  $\overline{b}_{\lambda} \in \id{H}$ is a torsion element, and since
  the module $\id{H}$ is torsion-free, it follows that
  $\overline{b}_{\lambda} = 0$ and thus $b_{\lambda} \in
  \iota(A^-(\KL))$, as claimed.  We have $b_{\lambda} =
  \iota(\gamma')$ for some $\gamma' \in A^-(\KL)$ and
\[ T a = T a_{\lambda} + T \omega_M a_{\mu} = \id{N}( T b ) = \id{N}
(b_{\lambda} + b_{\mu}) = p \gamma' + (T \omega_M a_{\mu} + p c), \]
for some $c \in C^-(\KL')$. Hence, after canceling terms, we obtain
\[ T a_{\lambda} = p (\gamma' + c), \quad \hbox{and} \quad p \gamma =
T T^* a_{\lambda} \in \Lambda a_{\lambda}, \quad \hbox{for $\gamma :=
  T^* \gamma'$}. \]  
  
We raise a contradiction by showing that this identity is inconsistent
at finite levels.  Since $p \gamma = T T^* a_{\lambda}$ as coherent
sequences, the identity holds a fortiori at level $M$.  Letting $g(T)
= T T^*$, we notice from the definition that $a_M = a_{\lambda,M}$ and
$T T^* h_M = N_{\K_M/\rg{k}_M}( T T^* a_M) \neq 0$. The resulting
identity $ p \gamma_M = g(T) a_{\lambda, M} = g(T) a_M$ satisfies
premises of Corollary \ref{lowpart}. Since $\deg(g) = 2 < p$, it
follows that $\gamma_M \in \Lambda a_M$, and taking norms,
$N(\gamma_M) \in \Lambda h_M$ -- where we write $N = N_{\K/\rg{k}}$.
Let now $\beta_M = N(b_M) \in A^-(L'_M)$ and $y_M = g_2(T) h_M =
N(\gamma_M) \in A^-(L_M), z_M = N(b_{\mu,M})$. By definition, we have
$p z_M = 0, p \beta_M = h_M$ and consequently
\[ T p \beta_M = T h_M = p ( g(T) h_M + z_M) = p T T^* h_M \quad
\Rightarrow \quad T ( 1 - p T^* ) h_M = 0.\] The last identity implies
$T h_M = 0$ and thus $\prk(\Lambda h_M) = 1$. This contradicts the
fact that $M$ was chosen such that $\prk(\Lambda h_M) = \lambda^-(L) =
p(p-1) > 1$, showing that the extensions, $\KL, \KL'$ cannot exist and
confirms the claim of Theorem \ref{main}.
\end{proof}

\section{Appendix : Proof of Proposition \ref{fantom}}
Let $N = A^-(T^*)$ be defined in the cyclotomic $\Z_p$-extension of
$\K$, and suppose that $\K$ is a CM-extension with positive Leopoldt
defect and containing the \nth{p} roots of unity. We have mentioned
that $\zprk(\Gal(\M^+/\K_{\infty})) = \id{D}(\K)$, a fact which is
proved in all text-books.

Let $\M \subset \Omega(\K)$ be the product of all $\Z_p$-extensions of
$\K$ and let $\KP = \M^- \cap (\KH^- \cap \Omega_E)$. One can build an
explicit map $\rho : E(\K) \otimes_{\Z} \Z_p \ra \overline{E}$ such
that $\Ker(\rho) \sim \Rad(\KP/\K_{\infty})$ and thus
$\zprk(\Gal(\KP/\K_{\infty})) = \id{D}(\K)$ while reflection yields
$T^* \Gal(\KP/\K_{\infty}) = 0$.  The extension $\KP$ was for instance
investigated by Jaulent in \cite{Ja}; we denote it \textit{the phantom
  field} associated to the Leopoldt conjecture.

The $p$-ramified, $p$-abelian, real extensions of $\K_{\infty}$ are
obtained as Kummer extensions by taking roots of classes in $A^-$,
according to the point 2. in Remark \ref{shift}. In fact, if $\Omega^+
= \Omega^{X^-}$ is the fixed field of the minus part of $X :=
\Gal(\Omega/\K_{\infty})$, we have $\Omega^+ = \overline{\K}_{\infty}[
(A^-)\pinf ]$. Here $\overline{\K}$ indicates that we might have to
adjoin first the roots of some expressions of the type
$\wp/\overline{\wp}$, with $\wp$ a principal prime of $\K_{\infty}$
above $p$.

The galois properties of the Kummer pairing imply more precisely that
$\M^+ \subset \K_{\infty}[ N\pinf ]$, and since $T \Gal(\M^+ /
\K_{\infty} ) = 0$, it follows that $T^* \Rad(\M^+ /\K_{\infty} ) =
0$. Considering $Y := \Gal(\K_{\infty}[ N\pinf ]/\K_{\infty}) \cong
N^{\bullet}$, it follows in fact that $\M^+ = (\K_{\infty}[ N\pinf
])^{T Y}$. By duality it follows that $\Rad(\M^+ /\K_{\infty} ) \cong
N/(T^* N)$. There is an exact sequence of pseudoisomorsphisms:
\[ 1 \ra N[ T^* ] \ra N \ra N \ra N/(T^* N) \ra 1, \]
in which the central map is $T^* : N \ra N$. From this, we deduce 
\begin{eqnarray*}
  \id{D}(\K) & = & \zprk(\Gal(\M^+/\K_{\infty})) = \zprk(\Rad(\M^+/\K_{\infty})) \\ & = & \eprk(N/T^* N) =
  \eprk(N[ T^* ]) . 
\end{eqnarray*}
We have thus shown that $\eprk(A^-[ T^* ]) = \id{D}(\K)$ and for each
$\KL \subset \M^+$ there is a sequence $a \in A^-(T^*) \setminus T^*
A^-(T^*)$ with $\LK = \K_{\infty}[ a\pinf ]$. This completes the proof
of the Proposition \ref{fantom}.

The following useful fact was proved by Sands in \cite{Sd}:
\begin{lemma}
 \label{sa}
 Let $\KL/\K$ be a $\Z_p$-extension of number fields in which all the
 primes above $p$ are completely ramified. If $F(T) \in \Z_p[ T ]$ is
 the minimal annihilator polynomial of $L(\KL)$, then $(F, \nu_{n,1})
 = 1$ for all $n > 1$.
\end{lemma}

\begin{Acknow}
  \nonumber This is an alternative approach to several previous
  attempts which used $\lambda$-type rather than $\mu$-type sequences;
  with that approach it was only possible to prove the case of the
  Leopoldt conjecture, in which $p$ is totally split in $\K$. The new
  approach grew from discussions with S\"oren Kleine, related to his
  PhD thesis that concerns Greenberg's Null Space Conjecture. It is to
  an important extent due to the involvement of Kleine with
  $\mu$-extension and the discussions had with him
  on related subjects, that made the elaboration of this proof
  possible. 

  I thank Cornelius Greither for his careful reading of preliminary
  drafts of this version. His remarks and the questions he asked
  during a period of almost one year, helped to substantially improve
  the quality of the paper and eliminate several flaws.  His valuable
  remarks lead to improvements of the text; in particular the proof of
  the first two conditions in Lemma 7 is due to him.

  I would like to express my gratitude to the precious few who
  assisted earlier attempts with discussions and comments: Grzegorz
  Banaszak, John Coates, Ralf Greenberg, Hendrik Lenstra, Florian
  Pop. Ina Kersten and the colleagues at the Mathematical Institute of
  G\"ottingen have provided during many years of evolution, an
  atmosphere of understanding and encouragement, which was supportive
  for long time research: to them my sincere gratefulness.

  Last but not least, this is to Theres and Seraina, who indulged over
  years with an ambig family presence of the researcher.
\end{Acknow}

\bibliographystyle{abbrv} 
\bibliography{leoMuV8}
\end{document}